\documentclass[letterpaper,11pt]{amsart}

\usepackage[all]{xy}                        %

\CompileMatrices                            

\UseTips                                    

\input xypic
\usepackage[bookmarks=true]{hyperref}       

\usepackage{amssymb,latexsym,amsmath,amscd}
\usepackage{xspace}

\usepackage{graphicx}
\usepackage{color}

\reversemarginpar

\theoremstyle{plain}
\newtheorem{theorem}{Theorem}[section]
\newtheorem*{theorem*}{Theorem}
\newtheorem{proposition}[theorem]{Proposition}

\newtheorem{lemma}[theorem]{Lemma}

\theoremstyle{definition}

\newtheorem{remark}[theorem]{Remark}
\newtheorem{example}[theorem]{Example}

\newcommand{\enm}[1]{\ensuremath{#1}}          %
\newcommand{\op}[1]{\operatorname{#1}}
\newcommand{\cal}[1]{\mathcal{#1}}

\newcommand{\ZZ}{\enm{\mathbb{Z}}}

\newcommand{\PP}{\enm{\mathbb{P}}}

\newcommand{\Aa}{\enm{\cal{A}}}

\newcommand{\Ee}{\enm{\cal{E}}}
\newcommand{\Ff}{\enm{\cal{F}}}
\newcommand{\Gg}{\enm{\cal{G}}}
\newcommand{\Hh}{\enm{\cal{H}}}
\newcommand{\Ii}{\enm{\cal{I}}}

\newcommand{\Ll}{\enm{\cal{L}}}

\newcommand{\Nn}{\enm{\cal{N}}}
\newcommand{\Oo}{\enm{\cal{O}}}

\newcommand{\Uu}{\enm{\cal{U}}}

\renewcommand{\phi}{\varphi}
\renewcommand{\theta}{\vartheta}
\renewcommand{\epsilon}{\varepsilon}

\newcommand{\Ext}{\op{Ext}}

\renewcommand{\to}[1][]{\xrightarrow{\ #1\ }}

\newcommand{\old}[1]{}

\begin{document}

\title[Globally generated vector bundles]{Globally generated vector bundles on $\PP^1 \times \PP^1 \times \PP^1$ with low first Chern classes}
\author{E. Ballico, S. Huh and F. Malaspina}
\address{Universit\`a di Trento, 38123 Povo (TN), Italy}
\email{edoardo.ballico@unitn.it}
\address{Sungkyunkwan University, Suwon 440-746, Korea}
\email{sukmoonh@skku.edu}
\address{Politecnico di Torino, Corso Duca degli Abruzzi 24, 10129 Torino, Italy}
\email{francesco.malaspina@polito.it}
\thanks{The first and third authors are partially supported by MIUR and GNSAGA of INDAM (Italy). The second author is supported by Basic Science Research Program 2010-0009195 through NRF funded by MEST. The third author is supported by the framework of PRIN 2010/11 \lq Geometria delle variet\`a algebriche\rq, cofinanced by MIUR}

\begin{abstract}
We classify globally generated vector bundles on $\PP^1 \times \PP^1 \times \PP^1$ with small first Chern class, i.e. $c_1= (a_1, a_2, a_3)$, $a_i \le 2$. Our main method is to investigate the associated smooth curves to globally generated vector bundles via the Hartshorne-Serre correspondence.
\end{abstract}

\maketitle

\section{Introduction}
Globally generated vector bundles on projective varieties play an important role in classical algebraic geometry. If they are non-trivial they must have strictly positive first Chern class. The classification of globally generated vector bundles with low first Chern class has been done over several rational varieties such as projective spaces \cite{am,SU} and quadric hypersurfaces \cite{BHM}. There is also a recent work over complete intersection Calabi-Yau threefolds and a Segre threefold $\PP^1 \times \PP^2$ by the authors \cite{BHM+++, BHM++++}.

There are three types of Segre varieties of dimension $3$: $\PP^3$, $\PP^1 \times \PP^2$ and $\PP^1 \times \PP^1 \times \PP^1$. In this paper we examine the similar problem of classification of globally generated vector bundles for the Segre variety $\PP^1 \times \PP^1 \times \PP^1$, the product of three projective lines. Note that the classification is already dealt in the case of $\PP^3$ and $\PP^1 \times \PP^2$ in \cite{am,BHM++++,SU}.

The Hartshorne-Serre correspondence states that the construction of vector bundles of rank $r$ at least $2$ on a smooth variety $X$ with dimension $3$ is closely related with the structure of curves in $X$ and it inspires the classification of vector bundles on smooth projective threefolds. There have been several works on the classification of {\it arithmetically Cohen-Macaulay} (ACM) bundles on the Segre threefold \cite{CFM0,CFM} and so it is sufficiently timely to classify the globally generated vector bundles on the Segre threefolds.

Our first main result is on rank $2$ bundles on $\PP^1 \times \PP^1 \times \PP^1$:

\begin{theorem}\label{thm1}
Let $\Ee$ be an indecomposable and globally generated vector bundle of rank $r$ at least $2$ on $X=\PP^1 \times \PP^1 \times \PP^1$ with the Chern classes $c_1=(a_1, a_2, a_3)$ and $c_2=(e_1, e_2, e_3)$. Let $s$ be the number of connected components of associated curve to $\Ee$ via the Hartshorne-Serre correspondence. If $c_1\in \{(1,1,1), (2,1,1), (2,2,1)\}$, the quadruple $(s;e_1, e_2, e_3)$ and the possible rank $r$ are as follows:
\begin{enumerate}
\item [$\mathbf{(r=2)}$] \begin{itemize}
\item [(i)] $c_1(\Ee)=(2,1,1):$ up to permutations on $(e_2, e_3)$
\begin{align*}
\hspace{40pt}\{(3;0,3,3), (2;0,2,2), (1;2,1,1), (1;1,1,1), (1;1,2,0)\};
\end{align*}
\item [(ii)]$c_1(\Ee)=(2,2,1):$ up to permutations on $(e_1, e_2)$
\begin{align*}
\hspace{40pt}\{(1;2,1,2),(1;3,1,2),(1;4,1,2), (2;2,0,4), (3;3,0,6)\}.
\end{align*}
\end{itemize}

\item [$\mathbf{(r\ge 3)}$] \begin{itemize}
\item[(iii)] $c_1(\Ee)=(1,1,1):$ $\{(1;1,1,1~; 3 \le r \le 7)\}$;
\item [(iv)]$c_1(\Ee)=(2,1,1):$ up to permutations on $(e_2, e_3)$
\begin{align*}
\{ &(3;0,3,3~; 3\le r \le 4), (1;2,2,2~;3\le r\le 5),\\
& (1;2,3,3~;3\le r \le 8), (1;2,4,4~;3\le r \le 11)\\
&(1;1,a,b~; 3\le r\le a+b)~|~3\le a+b\}.
\end{align*}
\end{itemize}
\end{enumerate}
Moreover there exist globally generated vector bundles in each case.
\end{theorem}

Since bundles with $c_1=(a_1, a_2, 0)$ are pull-backs of bundles on $Q=\PP^1 \times \PP^1$ with the first Chern class $(a_1, a_2)$ by Proposition \ref{trivial}, so Theorem \ref{thm1} gives us a complete answer for the first Chern class $c_1<(2,2,2)$ with $a_1\ge a_2 \ge a_3 \ge 0$. Indeed we give a complete classification of vector bundles and associated curves with respect to $5$-tuple $(s;e_1, e_2, e_3; r)$ in most cases (see Proposition \ref{a1}, Proposition \ref{a2.1++}, Theorem \ref{prop4.2} and Theorem \ref{kkkk2}). We also give a partial classification of globally generated vector bundles of rank $2$ on $X$ with $c_1=(2,2,2)$ in Section $6$.

Let us here summarize the structure of this paper. In Section $2$, we introduce the definitions and main properties that will be used throughout the paper, mainly the Hartshorne-Serre correspondence that relates the globally generated vector bundles with smooth curves contained in the Segre threefolds. In Section $3$, we collect several basic notations and techniques that are used throughout the article, and then we give a complete classification of globally generated vector bundles of arbitrary rank with $c_1=(1,1,1)$ as a warm-up. Together with the results in Section $4 \sim 6$, we complete the classification of globally generated vector bundles of rank $2$ with $c_1\le (2,2,2)$ and also give the classification for arbitrary rank in the case of $c_1\le (2,1,1)$.


\section{Preliminaries}

Let $V_1, V_2, V_3$ be three $2$-dimensional vector spaces with the coordinates $[x_{1i}], [x_{2j}], [x_{3k}]$ respectively with $i,j,k\in \{1,2\}$. Let $X\cong \PP (V_1) \times \PP (V_2) \times \PP (V_3)$ and then it is embedded into $\PP^7\cong \PP(V_0)$ by the Segre map where $V_0=V_1 \otimes V_2 \otimes V_3$.

The intersection ring $A(X)$ is isomorphic to $A(\PP^1) \otimes A(\PP^1) \otimes A(\PP^1)$ and so we have
$$A(X) \cong \ZZ[t_1, t_2, t_3]/(t_1^2, t_2^2, t_3^2).$$
We may identify $A^1(X)\cong \ZZ^{\oplus 3}$ by $a_1t_1+a_2t_2+a_3t_3 \mapsto (a_1, a_2, a_3)$. Similarly we have $A^2(X) \cong \ZZ^{\oplus 3}$ by $e_1t_2t_3+e_2t_3t_1+e_3t_1t_2\mapsto (e_1, e_2, e_3)$ and $A^3(X) \cong \ZZ$ by $ct_1t_2t_3 \mapsto c$. Then $X$ is embedded into $\PP^7$ by the complete linear system $|\Oo_X(1,1,1)|$ as a subvariety of degree $6$ since $(1,1,1)^3=6$.

Here we introduce some basic maps for our later use.
\begin{itemize}
\item $\pi_i : X \to \PP^1$ is the natural projection to $i^{\mathrm{th}}$ factor;
\item $\pi_{ij}: X \to \PP^1 \times \PP^1$ is the natural projection to $(i,j)$-factor;
\item $\phi=\phi_W:X \to \PP^3$ is a linear projection to $\PP^3$ from a $3$-dimensional subspace $W\subset \PP^7$ with $W\cap X=\emptyset$.
\end{itemize}

For a curve $C\subset X$, write $C = C_1\sqcup \cdots \sqcup C_s$ with $s\ge 1$ and $C_1,\dots ,C_s$ the connected components of $C$. Set
$$e_1= \deg (\Oo _C(1,0,0))~,~e_2= \deg (\Oo _C(0,1,0))~,~e_3= \deg (\Oo _C(0,0,1))$$
and call $(e_1, e_2, e_3)$ the {\it multidegree} of $C$. For each $i=1,\dots ,s$, we also set
$$e[i]_1:= \deg (\Oo _{C_i}(1,0,0))~,~e[i]_2:= \deg (\Oo _{C_i}(0,1,0))~,~e[i]_3:= \deg (\Oo _{C_i}(0,0,1)).$$
We also set $\deg (C) := C \cdot \Oo _X(1)$ and call it the {\it degree} of $C$. Then $\deg (C)$ is the degree of $C$ as a curve in $\PP^7$ and it is the sum of the factors of multidegree of $C$.

For a coherent sheaf $\Ee$ with the second Chern class $c_2(\Ee )=e_1t_2t_3+e_2t_3t_1+e_3t_1t_2$, we say that $c_2(\Ee )=(e_1,e_2,e_3)$ or that $\Ee$ has multidegree $(e_1,e_2,e_3)$. Let $\Ee$ be a globally generated vector bundle of rank $r$ on $X$ with the first Chern class $c_1(\Ee)=(a_1, a_2, a_3)$. Then it fits into the exact sequence
\begin{equation}\label{equ+}
0\to \Oo_X^{\oplus (r-1)} \to \Ee \to \Ii_C (a_1, a_2, a_3) \to 0,
\end{equation}
where $C$ is a smooth subscheme of dimension $1$ on $X$ by \cite[Section $2$. $\mathbf{G}$]{man}. If $C$ is empty, then $\Ee$ is isomorphic to $\Oo_X^{\oplus (r-1)} \oplus \Oo_X(a_1, a_2, a_3)$.

\begin{proposition}\cite{sierra}\label{prop1}
If $\Ee$ is a globally generated vector bundle of rank $r$ on $X$ with the first Chern class $c_1$ such that $H^0(\Ee(-c_1))\not= 0$, then we have $\Ee\simeq \Oo_X^{\oplus (r-1)}\oplus \Oo_X(c_1)$.
\end{proposition}

In particular, in the classification of globally generated vector bundles on $X$, we may assume that $C$ is not empty and $H^0(\Ee(-c_1))=0$.

We recall from \cite{BHM++++} the following elementary observations.

\begin{proposition}\cite[Proposition 2.3]{BHM++++}\label{trivial}
For $c_1=(a,b,0)\in \ZZ_{\ge0}^{\oplus 3}$, there is a bijection $\Ee \mapsto \pi_{12}^*(\Ee)$ between the set of spanned vector bundles $\Ee$ of rank $r$ on $\PP^1 \times \PP^1$ with $c_1(\Ee)=(a,b)$ and the spanned vector bundles of rank $r$ on $X$ with the first Chern class $c_1$. Moreover we have
\begin{enumerate}
\item $h^i (\PP^1 \times \PP^1, \Ee) = h^i (X, \pi_{12}^* (\Ee))$ for all $i\ge 0$;
\item for any spanned bundle $\Gg$ on $X$ with $c_1(\Gg)=(a,b,0)$, we have $\Gg \cong \pi_{12}^*({\pi_3}_*(\Gg))$ with ${\pi_{12}}_* (\Gg)$ a spanned bundle on $\PP^1 \times \PP^1$.
\end{enumerate}
\end{proposition}

Notice that we have $c_3({\pi_3} ^\ast (\Ee )) =0$ even when $r\ge 3$. By Proposition \ref{trivial} it is now sufficient to check the case of $c_1=(a_1,a_2,a_3)$ with $a_i>0$ for all $i$.

\begin{remark}\label{24nov}
Assume that $C$ is a curve with $\omega _C \cong \Oo _C(c_1-c_1(X))$. If $C$ has $s$ connected components, then we have $h^0(\omega _C(c_1(X)-c_1)) =s$ and so the Hartshorne-Serre correspondence shows that $C$ gives a vector bundle $\Ee$ with $c_1(\Ee)=c_1$ of rank $r$ with no trivial factor if and only if $2 \le r \le s+1$.
\end{remark}

\begin{remark}\label{uuu00}
On a smooth threefold $X$, let us fix a very ample line bundle $\Ll$ and a smooth curve $C\subset X$. Assume that $\Ii _C\otimes \Ll$ is globally generated and take two general divisors $M_1,M_2\in |\Ii _C\otimes \Ll|$. Set $Y:= M_1\cap M_2$. Since $\Ll$ is very ample, each connected component of $C$ appears with multiplicity $1$
in the locally complete intersection curve $Y$ By the Bertini theorem we have $Y =C\cup D$ with either
$D=\emptyset$ or $D$ a reduced curve containing no component of $C$ and smooth outside $C\cap D$.
\end{remark}

\begin{example}\label{ooo0}
On a smooth and connected projective threefold $X$, let us fix a globally generated line bundle $\Ll$ with $h^0(\Ll )\ge 2$ and set $r_0:= h^0(\Ll )-1$. Since $\Ll$ is globally generated, so the evaluation map $\psi: H^0(\Ll)\otimes \Oo _X\to \Ll$ is surjective and $\ker (\psi )$ is a vector bundle of rank $r_0$ on $X$. The vector bundle $\Ff := \ker (\psi )^\vee$ fits in an exact sequence
\begin{equation}\label{oo1}
0 \to \Ll ^\vee \to \Oo _X^{\oplus (r_0+1)} \to \Ff \to 0
\end{equation}
and it determines the Chern classes of $\Ff$. If $h^1(\Ll ^\vee )=0$, e.g. $\Ll$ is ample, then the sequence (\ref{oo1}) gives $h^0(\Ff ) =r_0+1$. If $Y\subset X$ is the complete intersection of two elements of $|\Ll|$, then we get $Y$ as the dependency locus of a certain $(r_0-1)$-dimensional linear subspace of $H^0(\Ff)$.
\end{example}

\begin{example}\label{ooo1}
Let us apply the construction in Example \ref{ooo0} to $X=\PP^1 \times \PP^1 \times \PP^1$.
If $\Ll := \Oo _X(a_1,a_2,a_3)$ with $a_i \ge 0$ for all $i$ and $a_1+a_2+a_3>0$, then we have $ r_0+1:=h^0(\Ll)=(a_1+1)(a_2+1)(a_3+1)$. Let $Y$ be the complete intersection of two elements of $|\Oo _X(a_1,a_2,a_3)|$. Since $(a_1t_1+a_2t_2+a_3t_3)^2= 2a_1a_2t_1t_2+2a_1a_3t_1t_3+2a_2a_3t_2t_3$, $Y$ has multidegree $(2a_2a_3,2a_1a_3,2a_1a_2)$. By the adjunction formula, we also have $\omega _Y \cong \Oo _Y(2a_1-2, 2a_2-2, 2a_3-2)$ and so $\omega _Y(2-a_1,2-a_2,2-a_3) = \Oo _Y(a_1,a_2,a_3)$. In particular $\omega _Y(2-a_1,2-a_2,2-a_3)$ is spanned and we have $h^0(\omega _Y(2-a_1,2-a_2,2-a_3))= r_0-1$.

Now we assume that $Y$ is smooth, e.g. take as $Y$ the complete intersection of two general elements of $|\Oo _X(a_1,a_2,a_3)|$. Since $\dim (Y)=1$, and $\omega _Y(2-a_1,2-a_2,2-a_3)$ is a spanned line bundle, so $\omega _Y(2-a_1,2-a_2,2-a_3)$ is spanned by a general $m$-dimensional linear subspace of $H^0(\omega _Y(2-a_1,2-a_2,2-a_3))$ for every integer $m$ with $2\le m \le r_0-1$. The Hartshorne-Serre correspondence gives the existence of a globally generated vector bundle $\Ee$ with $Y$ as a dependency locus and no trivial factor (see \cite[Lemma 4.1]{BHM++}) for all ranks $r$ with $3 \le r \le r_0$.

If $a_i> 0$ for all $i$, i.e. $\Oo _X(a_1,a_2,a_3)$ is ample, then a standard exact sequence and a vanishing theorem give that $h^0(\Oo _Y)=1$ and in particular each $Y$ is connected. The same proof works if  and only if at least two among $a_1, a_2, a_3$ are positive or if $(a_1, a_2, a_3)=(1,0,0)$. In the case $r = r_0$, Example \ref{ooo0} shows that there is a unique bundle $\Ee$ with rank $r_0$ and associated to some complete intersection curve. Since $\Ee$ is unique, we have $g^\ast (\Ee )\cong \Ee$ for each $g\in \mathrm{Aut} ^0(X) = \mathrm{Aut} (\PP^1)\times \mathrm{Aut} (\PP^1)\times \mathrm{Aut} (\PP^1)$, i.e. $\Ee$ is homogeneous for the action of $\mathrm{Aut} ^0(X)$. If $a_1=a_2=a_3$, then $\Ee$ is homogeneous for the action of $\mathrm{Aut}(X)$.
 \end{example}

\begin{remark}\label{rem3.3.3}
Let $Y \subset X$ be the complete intersection of two divisors of type $|\Oo _X(a_1, a_2, a_3)|$ containing $C$. We have $\deg (Y)=2(a_1a_2+a_2a_3+a_3a_1)$, where $\deg(Y)$ is the degree of $Y$ as a curve in $\PP^7$. By the Bertini theorem $Y$ is a curve containing $C$ and smooth outside $C$. Note that $C$ occurs with multiplicity one in $Y$, because $\Ii _C(a_1, a_2, a_3)$ is spanned and so, affixing $p_i\in C_i$, $1\le i \le s$, we may find a divisor $T\in |\Ii _C(a_1, a_2, a_3)|$ not containing the tangent line of $C_i$ at $p_i$. $Y$ is also connected since we have $h^0(\Oo_Y)=1$ by a vanishing theorem. The adjunction formula gives $\omega _Y \cong \Oo _Y(2a_1-2,2a_2-2,2a_3-2 )$ and so we have
\begin{align*}
2p_a(Y) -2& = (a_1t_1+a_2t_2+a_3t_3)^2((2a_1-2)t_1+(2a_2-2)t_2+(2a_3-2)t_3) \\
&= 12a_1a_2a_3-4(a_1a_2+a_2a_3+a_3a_1).
\end{align*}
Hence we have $p_a(Y)=6a_1a_2a_3-2(a_1a_2+a_2a_3+a_3a_1)+1$.
\end{remark}

\begin{remark}\label{tr1}
Let $D$ be an integral projective curve. By the universal property of $\PP^1$ there is a bijection between the morphisms $u: D\to X$ and the triples $(u_1,u_2,u_3)$ with $u_i: D \to \PP^1$ any morphism. The set $u(D)$ is contained in a $2$-dimensional factor of $X$ if and only if one of the $u_1,u_2,u_3$ is constant. We say that a constant map has degree zero. With this convention to any $u$ we may associate a triple $(\deg (u_1),\deg (u_2),\deg (u_3))\in \ZZ_{\ge 0}^{\oplus 3}$ and $u(D)$ is a curve if and only if $(\deg (u_1),\deg (u_2),\deg (u_3)) \ne (0,0,0)$. Now assume that $u$ is birational onto its image. With this assumption for all $(a_1,a_2,a_3)\in \ZZ^3$ we have $u(D)\cdot \Oo _X(a_1,a_2,a_3) =a_1\deg (u_1)+a_2\deg (u_2)+a_3\deg (u_3)$. In particular the degree of the curve $u(D)$ is $\deg (u_1)+\deg (u_2)+\deg (u_3)$.
\end{remark}

Now let us collect several observation concerning the case $c_1=(a,b,1)$ with $a,b>0$.
\begin{lemma}\label{f1}
If $\Ii_C(a,b,1)$ is globally generated, then the map ${\pi _{12}}_{\vert_C}: C \to \PP^1 \times \PP^1$ is an embedding.
\end{lemma}
\begin{proof}
For a point $p\in \PP^1\times \PP^1$, set $J: = \pi _{12}^{-1} ( p)$. Assume for the moment that $J$ is a connected component of $C$. Since $\Oo _J(2-a,2-b,-1)$ has degree $-1$ and $\omega _J$ has degree $-2$, so $\omega _J(2-a,2-b,-1)$ is not spanned and in particular $J$ is not a component of $C$. Since $\Ii _C(a,b,1)$ is globally generated and $\deg (\Oo _J(a,b,1)) =1$, so we have $\deg (J \cap C)\le 1$ and the assertion.
\end{proof}

\begin{remark}\label{f3}$ $
\begin{enumerate}
\item If $Y$ is the complete intersection of two general elements of $\Ii _C(a,b,1)$, the curve $C$ is contained in $Y$, and in particular we have $e_1\le 2b$, $e_2\le 2a$ and $e_3\le 2ab$.
\item Assume $s\ge 2$. Since the map ${\pi _{12}}_{\vert_C}$ is an embedding, there is an integer $j\in \{1,2\}$ such that $e[i]_j =0$ and $e[i]_{3-j}=1$ for all $i$. Let us assume $j=1$ without loss of generality. Each $C_i$ is smooth and rational and so we have
$\omega _{C_i} \cong \Oo _{C_i}(a-2,b-2,-1)$ (resp. $\omega _{C_i}(2-a,2-b,1)$ is spanned but not trivial) if and only if $b-2-e[i]_3 = -2$, i.e. $e[i]_3 =b$ (resp. $e[i]_3>b$).
\end{enumerate}
 \end{remark}


 \section{Warm-up and Case of $c_1=(1,1,1)$}

Let us deal with the globally generated vector bundles with $c_1=(1,1,1)$. In this case $Y$ has multidegree $(2,2,2)$ and $p_a(Y) =1$. Since $Y = C\cup D$ is connected,
we have $p_a(C_i)=0$ for all $i$, unless $C=Y$. Since $\omega _C(1,1,1)$ is globally generated, no component of $C$ is a line.

\begin{proposition}\label{a1}
Let $\Ee$ be a globally generated vector bundle of rank $2$ on $X$ with $c_1=(1,1,1)$ and no trivial factor. Then its associated curve $C$ is a smooth conic and $\Ee$ is isomorphic to $\Oo_X(1,0,0)\oplus \Oo_X(0,1,1)$, up to permutation of factors.
\end{proposition}
\begin{proof}
Since $\Ee$ is assumed to have no trivial factor, then $C\ne \emptyset$. From the sequence (\ref{equ+}) we have $\omega_C \cong \Oo_C(-1)$ and so each connected component of $C$ is a smooth conic, i.e. it is a fiber of one of $\pi_i$. Let $s$ be the number of connected components of $C$ and write $C=C_1 \sqcup \cdots \sqcup C_s$ with each $C_i$ a smooth conic. Since the sheaf $\Ii _C(1)$ is spanned, we have $\deg ({C})\le 6$. The equality holds if and only if $C$ is the complete intersection of $X$ with two hyperplane sections, which is not possible since we would have $\omega _C\cong \Oo_C$ from the minimal free resolution of $\Ii_C$. Thus we have $1\le s \le 2$.

If $s=1$, $C$ is a connected and smooth conic and so it is a hyperplane section of a fiber of a projection. For instance, if $C=\{o_1\} \times C'$ with $C'$ a smooth conic in $\PP^1 \times \PP^1$, we can take $[x_{11}, x_{12}]$ so that $o_1=[1,0]$. Then, using that $\{o_1\} \times \PP^1 \times \PP^1$ is defined by the $4$ linear equations
$$\{ x_{12}x_{21}= x_{12}x_{22}=x_{12}x_{31}=x_{12}x_{32}=0\},$$
we get that $\Ii_C(1)$ is spanned. Since $\Oo _X(1,0,0)\oplus \Oo _X(0,1,1)$ has $(1+t_1)(1+t_2+t_3) = 1+(t_1+t_2+t_3)+(t_1t_2+t_1t_3)$, it is associated to a connected curve $C'$ of multidegree $(0,1,1)$. Take $f\in \mathrm{Aut}(X)$ with
$f({C}) = C'$ and get that $\Ee \cong f^*(\Oo _X(1,0,0)\oplus \Oo _X(0,1,1)) \cong \Oo _X(1,0,0)\oplus \Oo _X(0,1,1)$.

Now assume $s=2$ and then we have $e[1]_1+e[1]_2+e[1]_3=e[2]_1+e[2]_2+e[2]_3=2$. Since ${\pi _{12}}_{\vert_C}$ is an embedding, $\pi_{12}({C})$ is the disjoint union of two lines
of $\PP^1\times \PP^1$ and so we have either $e[1]_1=e[2]_1 =0$ or $e[1]_2=e[2]_2 =0$. With no loss of generality we may assume
$e[1]_2=e[2]_2=0$. Since each $\pi _{13}(C_i)$ is a curve of type $(1,1)$, we get $\pi _{13}(C_1)\cap \pi _{13}(C_2)\ne \emptyset$ and so ${\pi _{13}}_{\vert_C}$ is not
an embedding, a contradiction.
\end{proof}
Now we consider the higher rank case.

\begin{lemma}\label{a2++}
Let $C\subset X$ be a smooth curve such that $\omega _C(1,1,1)$ and $\Ii _C(1,1,1)$ are globally generated. If $\omega _C(1,1,1)$ is not trivial, then $C$ is connected and it is one of the following:
\begin{enumerate}
\item[(i)] a linearly normal elliptic curve of degree $6$ with multidegree $(2,2,2)$,
\item[(ii)] a normal rational curve of degree $3$ with multidegree $(1,1,1)$,
\item[(iii)] a normal rational curve of degree $4$ with multidegree $(2,1,1)$, up to permutations.
\end{enumerate}
In each case we have $h^0(\omega_C(1,1,1))=6,2$ and $3$, respectively.
\end{lemma}

\begin{proof}
We have $\deg ({C})\le 6$ and the equality holds if and only if $C=Y$, the complete intersection of two elements in $|\Oo_X(1,1,1)|$. Here $C$ has $s=1$ and multidegree $(2,2,2)$. Example \ref{ooo1}
gives $h^0(\omega _C(1,1,1)) +1 =7$.

From now on we assume $\deg ( C) \leq 5$. Let $Y$ be the intersection of two general elements of $|\Ii _C(1,1,1)|$. By Remark \ref{uuu00} we have $Y = C\cup D$ with $D$ a reduced curve, $C\cap D$ finite and $D$ smooth outside $C\cap D$. If $T$ is a smooth elliptic curve and $D$ is a reduced curve with $D\cap T \ne \emptyset$, then we have
${\omega _{T\cup D}}_{|_T} \not \cong \Oo _T$. Since $D\ne \emptyset$ and $\omega _Y\cong \Oo _Y$, each $C_i$ is smooth and rational. Since $\omega _C(1,1,1)$ is globally generated, no connected component of $C$ is a line. If $\deg ( C)=5$, then $D$ is a line. Since $p_a(Y)=1$, we have $\deg (D\cap C)\ge 2$ and so $D$ is in the base locus of $\Ii _C(1,1,1)$, a contradiction. Thus we have $\deg (C)\le 4$ and $s\le 2$ since no component of $C$ is a line. If $s=2$, then we have $\deg (C_i)=2$ for all $i$ and so $\deg (\omega_C(1,1,1))=0$. In particular we have $\omega_C(1,1,1) \cong \Oo_C$, contradicting our assumption.

From now on we assume $s=1$. We have $\deg ({C}) \ne 2$, because we assumed that $\omega _C(1,1,1)\ne \Oo _C$. Since $C$ is rational, we have $h^0(\omega _C(1,1,1)) =\deg ({C})-1$. Since $C$ is rational and ${\pi _{12}}_{\vert_C}$ is an embedding, so we have either $e_1=1$ or $e_2=1$. Similarly since ${\pi _{13}}_{\vert_C}$ and ${\pi_{23}}_{\vert_C}$ are embeddings, we have $e_1=1$ or $e_3=1$, and $e_2=1$ or $e_3=1$. In particular two of the integers $e_1,e_2,e_3$ are ones. Hence if $\deg ({C}) =3$, then $C$ has multidegree $(1,1,1)$, while if $\deg ({C})=4$, then $C$ has multidegree either $(2,1,1)$, $(1,2,1)$ or $(1,1,2)$. By symmetry one of them occurs if and only if all the three possibilities occur, but they give different families of bundles.
\end{proof}

\begin{remark}\label{ao.1}
Since all vector bundles of rank at least $2$ on $\PP^1$ are decomposable, Proposition \ref{trivial} shows that the decomposable vector bundles $\Ee$ without trivial factors, are obtained in the following way:

Take $i\in \{1,2,3\}$ and an integer $r$ such that there is a globally generated bundle $\Ff$ of rank $r-1$ on $\PP^1\times \PP^1\cong Q$ with $c_1(\Ff ) =(1,1)$ and no trivial factor. Take $\{j,k\}= \{1,2,3\}\setminus \{i\}$ with $j<k$ and set $\Ee:= \pi _i^{\ast}(\Oo _{\PP^1}(1))\oplus \pi _{jk}^\ast (\Ff)$. From \cite[Propositions 3.5 and 5.4]{BHM2} the possible $\Ff$ is as follows:
\begin{itemize}
\item[(1)] $\Oo_Q(1,0)\oplus \Oo_Q(0,1)$,
\item[(2)] $\phi_p^*T\PP^2(-1)$, where $\phi_p : Q \to \PP^2$ is the linear projection with the center $p\in \PP^3 \setminus Q$, or
\item[(3)] $T\PP^3(-1)_{\vert_{Q}}$.
\end{itemize}
In case (2) we have $h^0(\phi_p^*T\PP^2(-1)) = 3$ by \cite[proof of Proposition 3.5]{BHM2} and so $\Ff$ has rank $3$ and $h^0(\Ff ) =5$. It is the case $\mathrm{(iii)}$ of Proposition \ref{a2.1++}. On the other hand, we have $h^0(T\PP^3(-1)_{\vert_{Q}}) =4$ by \cite[proof of Proposition 5.4]{BHM2} and so $h^0(\Ff ) =6$. It is as in case $\mathrm{(iii)}$ of Proposition \ref{a2.1++} for $r=4$.\end{remark}

\begin{remark}\label{ao.1.1}
Take $\Ll = \Oo _X(1,1,1)$ in Remark \ref{ooo0}. Any bundle $\Ee$ of rank $r$ corresponds to an $(r+1)$-dimensional linear subspace of $H^0(\Oo _X(1,1,1))$ spanning $\Oo _X(1,1,1)$. Indeed if $r=7$, it gives $T\PP^7(-1)_{\vert_X}$, while if $3\le r \le 6$ it gives the bundles $\phi_W^*T\PP^r(-1)$ with $\phi_W: X \to \PP^r$ the restriction to $X$ of a linear projection from a linear subspace $W\subset \PP^7$ with $\dim (W) =6-r$ and $W\cap X = \emptyset$.
\end{remark}

\begin{proposition}\label{a2.1++}
Let $\Ee$ be a globally generated vector bundle of rank $r\ge 3$ on $X$ with $c_1=(1,1,1)$ and no trivial factor. If $C$ is a general dependency locus of $\Ee$, then $(p_a(C); e_1, e_2, e_3; r)$ are as follows:
\begin{itemize}
\item[(i)] $(1;2,2,2; 3 \le r \le 7)$; $C =Y$ and $\Ee$ is as in Remark \ref{ao.1.1}.
\item[(ii)]$(0;1,1,1; 3)$; $\Ee \cong \Oo _X(1,0,0)\oplus \Oo _X(0,1,0)\oplus \Oo _X(0,0,1)$.
\item[(iii)] $(0;2,1,1; 3 \le r \le 4)$, up to permutations on $(e_1, e_2, e_3)$; \\ $C$ is linearly normal in its linear span and $\Ee$ is as in Remark \ref{ao.1} $\mathrm{(2)}$ and $\mathrm{(3)}$ with $h^0(\Ee)=r+2$.
\end{itemize}
In case $(\rm i)$ for each $3\le r \le 7$ the bundles are parametrized by an irreducible family. In case $(\rm iii)$ for each $3\le r \le 4$ the bundles are parametrized by three irreducible families,
each of them corresponding to one of the possible multidegrees of $C$ and $h^0(\Ee ) = r+2$.
\end{proposition}

\begin{proof}
In the case $r=2$ we saw that there is no example with $s\ge 2$ and $\omega _C \cong \Oo _C$. Thus we only need to check which dependency loci with $\omega_C(1,1,1)$ spanned and $\omega _C(1,1,1) \not \cong \Oo _C$ arises for some bundle $\Ee$. Lemma \ref{a2++} gives a list of the potential curves $C$. If a certain $C$ exists, then it is dependency locus of a globally generated $\Ee$ of rank $r$ if and only if we have $3 \le  r \le h^0(\omega _C(1,1,1)) -1$. For the case $C=Y$, see Remark \ref{ao.1.1}.

\quad (a) Since $ \Ee_0:=\Oo _X(1,0,0)\oplus \Oo _X(0,1,0)\oplus \Oo _X(0,0,1)$ has $(1+t_1)(1+t_2)(1+t_3) = 1+(t_1+t_2+t_3) +(t_1t_2+t_1t_3+t_2t_3)+t_1t_2t_3$, it is associated
to a curve $C'$ of multidegree $(1,1,1)$. Lemma \ref{a2++} gives that $C'$ is connected. Take any $\Ee$ associated to a curve $C$ with multidegree $(1,1,1)$. There is $f\in \mathrm{Aut}(X)$ such that $f(C') =C$. Since $f^\ast \Ee_0 \cong \Ee_0$, $\Ee$ has rank $3$ and $h^0(\omega _C(1,1,1)) =2$, the Hartshorne-Serre correspondence gives $\Ee \cong \Ee_0$.

\quad (b) Let us check that $h^1(\pi _{23}^\ast (\phi _p^*T\PP^2(-1))^\vee(1,0,0)) =0$ and also that $h^1(\pi _{23}^*(\Omega _{\PP^3}(1)_{|Q})(1,0,0)) =0$, i.e. that there are no non-trivial extension of either $\pi _{23}^*(\phi_p^*T\PP^2(-1))$ or $\pi _{23}^*(T\PP^3(-1)_{\vert_{Q}})$ by $\Oo _X(1,0,0)$. Taking the pull-back first by by $\phi _p$ and then by
$\pi _{23}^*$ of the dual of the Euler's sequence of $T\PP^2$, we get the exact sequence
$$0 \to \Oo _X(0,-1,-1) \to \Oo _X(1,0,0)^{\oplus 3} \to \pi _{23}^*(\phi _p^*T\PP^2(-1)^\vee)(1,0,0) \to 0.$$
Since $h^2(\Oo _X(0,0,-1)) = h^1(\Oo _X(1,0,0)) =0$, we get
$$h^1(\pi _{23}^*(\phi _p^*T\PP^2(-1))^\vee(1,0,0)) =0.$$
Taking the pull-back by $\pi _{23}$ of the dual of the Euler sequence of $T\PP^3$, we get the exact sequence
$$0 \to \Oo _X(0,-1,-1) \to \Oo _X(1,0,0)^{\oplus 4} \to \pi _{12}^*(\Omega _{\PP^3}(1)_{|Q})(1,0,0)\to 0$$
and so we have $h^1(\pi _{23}^*(\Omega _{\PP^3}(1)_{|_Q})(1,0,0)) =0$.

 \quad ({c}) Take $C$ with multidegree $(2,1,1)$. Since $\deg (\Oo _C(0,1,1)) =2$, we have $h^0(\Oo _C(0,1,1)) =3$ and it implies $h^0(\Ii _C(0,1,1)) >0$ and so we get a non-zero map $u: \Oo _X(1,0,0) \to \Ee$. Since $h^0(\Ii _C(-1,1,1))=h^0(\Ii_C(0,0,1))=h^0(\Ii_C(0,0,1))=0$, so the sheaf $\mathrm{Im} (u)$ is saturated in $\Ee$ and $\Gg := \mathrm{coker} (u)$ is torsion-free. Since $\Ee$ is spanned, $h^1(\Oo _X(1,0,0)) =0$ and $h^0(\Ee )=r+2$, so the sheaf $\Gg$ has rank $r-1$ and $h^0(\Gg )=r$. Since $\Ee$ has no trivial factor, $\Gg$ has no trivial factor. Therefore $\Gg$ is the cokernel of a map $v: \Oo _X(0,-1,-1) \to \Oo _X^{\oplus r}$ with $v$ induced by an $r$-dimensional linear subspace $V$
of $H^0(\Oo _X(0,1,1))$. The map $v$ and the linear space $V$ correspond to a map $v': \Oo _{\PP^1\times \PP^1}(-1,-1) \to \Oo _{\PP^1\times \PP^1}^{\oplus r}$
and an $r$-dimensional linear subspace $V'$ of $H^0(\Oo _{\PP^1\times \PP^1}(1,1))$. $\Gg$ is locally free, i.e. it is as in the last two cases of Remark \ref{ao.1}, if and only if $\Gg '$ is locally free. Every $V' $ with no common zero defines a locally free $\Gg'$ and hence a locally free $\Gg$. Assume for the moment that $\Gg$ is locally free. Since $\Gg$ is spanned and
$\Ee$ is not a direct sum of three line bundles, so we get that $\Gg$ is either as in case (2) with $r=3$ or as in case (3) with $r=4$ of Remark \ref{ao.1}. Any $V'$ gives an injective map
$v'$ of sheaves, while every  $\Gg'$ is locally free if and only if $V'$ has no common zero. If $r=4$, then this is true, because $V =H^0(\Oo _{\PP^1\times \PP^1}(1,1))$ in this case.
Hence by case (3) of Remark \ref{ao.1} is the only bundle with $r=4$ in case (iii), while if $r=3$ case (2) of Remark \ref{ao.1} gives the only bundles $\Ee$ with $\Gg '$ locally free. Every $\Ee$ with $\Gg'$ not locally free is the flat limit of a family of bundles with $\Gg '$ locally free, i.e. of bundles as in case (2) of Remark \ref{ao.1}. To conclude the proof of Proposition \ref{a2.1++} it is sufficient to exclude the existence of $\Ee$ with $C$ of multidegree $(2,1,1)$ and which cannot be associated
to a locally free $\Gg$.

\quad {\emph{Claim $1$}}: For $\Ee$ with $r=3$ associated to $C$, we have $h^0(\Ee (-1,0,0)) =1$.

\quad {\emph {Proof of Claim $1$}}: It is sufficient to check that $h^0(\Ii _C(0,1,1)) =1$. This is true, because there are $o\in \PP^1$ and a unique $C' \in |\Oo _{\PP^1\times \PP^1}(1,1)|$ such that $C =\{o\}\times C'$. \qed

By \emph{Claim $1$}, $\Gg$ is uniquely determined by $\Ee$ and hence to complete the proof of Proposition \ref{a2.1++} it is sufficient to prove that $\Gg$ is locally free
for every  $\Ee$.

\quad {\emph {Claim $2$}}: For any $\Ee$ with $r=3$ we have $h^1(\Ee ^\vee )=1$.

\quad {\emph {Proof of Claim $2$}}: This is implicit in the Hartshorne-Serre correspondence. We have $\Ext^1(\Ii _C(1,1,1),\Oo _X) \cong H^0(\omega _C(1,1,1))^\vee$ and $\Ee$ is just given by a $2$-dimensional linear subspace $V$ of the $3$-dimensional space $H^0(\omega _C(1,1,1))^\vee$, while $H^0(\omega _C(1,1,1))^\vee/V$ represents $H^1(\Ee ^\vee )$. \qed

Let $\Ff$ be the only non-trivial extension of $\Ee$ by $\Oo _X$. By step (b) we have $\Ff \cong \Oo_X(1,0,0)\oplus \pi _{23}^*T\PP^3(-1)_{\vert_{Q}}$. Hence any $\Ee$
is the cokernel of a non-zero map $m: \Oo _X\to  \Oo_X(1,0,0)\oplus \pi _{23}^*T\PP^3(-1)_{\vert_{Q}}$. Write $m = (m_1,m_2)$ with $m_1\in H^0(\Oo _X(1,0,0))$ and $m_2\in H^0(\pi _{23}^*T\PP^3(-1)_{\vert_{Q}})$. We see that $\Gg$ is locally free if and only if $m_2$ has no common zero. Assume that $m_2$ has a common zero at $p =(p_1,p_2,p_3)\in \PP^1\times \PP^1\times \PP^1$ for some $\Ee$. We get that $m_2$ vanishes on $\PP^1\times \{(p_2,p_3)\}$.
Since $m_1$ has at least one zero on $\PP^1\times \{(p_2,p_3)\}$, we get a contradiction.
\end{proof}


\section{Case of $c_1=(2,1,1)$}

In this section we handle the globally generated vector bundles of rank $r$ on $X$ with $c_1=(a_1,a_2,a_3)$ with $a_i>0$ for all $i$ and $a_1+a_2+a_3=4$. Up to a change of the ruling
it is sufficient to do the case $c_1=(2,1,1)$ with the sequence (\ref{equ+}). Since we only look at bundles with no trivial factors, we have $C\ne \emptyset$. Note that $\omega_C(0,1,1)$ is globally generated and that $\omega _C\cong  \Oo_C(0,-1,-1)$ if $r=2$. Since $\omega _{C_i}(0,1,1)$ is spanned, so $C_i$ is not a line. By the case $b=1$ of Lemma \ref{f1} both $\pi _{12|_C}$ and $\pi _{13|_C}$ are embeddings. We have $(1+2t_1+t_2+t_3)(1+2t_1+t_2+t_3) =1+4t_1+2t_2+2t_3+4t_1t_2+4t_1t_3+2t_2t_3$. Therefore any scheme-theoretic intersection $Y$ of two elements in $|\Oo _X(2,1,1)|$ with $\dim (Y) = 1$ has multidegree $(2,4,4)$. We also have $h^0(\Oo _Y)=1$. Since $\omega _Y \cong \Oo_Y(2,0,0)$ by the adjunction formula, so we have $\deg (\omega _Y) = (4t_1t_2+4t_1t_3+2t_2t_3)(2t_1)  =4$, i.e. $Y$ has genus $3$, i.e.
$\Oo _X(2,1,1)$ has section genus $3$. Hence $C$ has multidegree $(e_1,e_2,e_3)$ with $e_1\le 2$, $e_2\le 4$, $e_3\le 4$ and $(e_1,e_2,e_3) =(2,4,4)$ if and only
if $C=Y$, i.e. $\Ee $ is as in Example \ref{ooo0}.

\begin{lemma}\label{salto1}
Assume $s>1$. Then we have
\begin{itemize}
\item[(i)] $2 \le s \le 3$
\item[(ii)] $e_1= 0$, $e[i]_2=e[i]_2= 1$ for all $i$, and
\item[(iii)] $\omega _C(0,1,1)\cong \Oo _C$.
\end{itemize}
\end{lemma}

\begin{proof}
Note that $\pi _{12|_C}$ is an embedding and any smooth curve of $\PP^1\times \PP^1$ with $s\ge 2$ connected components has either bidegree $(s,0)$ or bidegree $(0,s)$. Thus we have either $e_1=0$ or $e_2=0$. If $e_1=0$ (resp. $e_2=0$), then $e[i]_2=1$ (resp. $e[i]_1 =0$) for all $i$. Since $\pi _{13|_C}$ is an embedding, we get in the same way that either $e_1=0$ or $e_3=0$ and if $e_1=0$ (resp. $e_3=0$), then $e[i]_3=1$ (resp. $e[i]_1 =0$) for all $i$.

Assume for the moment $e_1>0$ and then we get $e_2 =e_3=0$, i.e. $C$ is the disjoint union of $s$ lines of multidegree $(1,0,0)$. Thus $\omega _{C_i}(1,0,0)$ has no non-zero section, a contradiction. Hence $e_1=0$,  $e[i]_2=e[i]_3= 1$ for all $i$, and $\omega _C(0,1,1)\cong \Oo _C$. Since $Y$ has multidegree $(2,4,4)$, we have $s\le 4$.
Assume $s=4$. We get that $D$ has multidegree $(2,0,0)$, i.e. that $D = D_1\sqcup D_2$ with each $D_i$ a line of multidegree $(1,0,0)$.
Since $\Ii _C(2,1,1)$ has not $D_j$ in the base locus, then $\deg (C\cap D_j) \le 2$. Since $C\cup D$ has $6$ irreducible components, each of them smooth and rational, we
get $p_a(Y) \le -1$, a contradiction.
\end{proof}

\begin{remark}\label{allspl}
Assume that $\Ee$ has no trivial factor. Then $\Ee$ is decomposable if and only if $\Ee$ is isomorphic to one of the following, up to reordering of the second and the third factors.
\begin{enumerate}
\item $\Oo_X(2,0,0)\oplus \Oo_X(0,1,1)$; $c_2(\Ee)=2t_1t_2+2t_1t_3$.
\item $\Oo_X(2,1,0)\oplus \Oo_X(0,0,1)$; $c_2(\Ee)=t_2t_3+2t_1t_3$.
\item $\Oo_X(1,1,1)\oplus \Oo_X(1,0,0)$; $c_2(\Ee)=t_2t_3+t_1t_3+t_1t_2$.
\item $\Oo_X(1,1,0)\oplus \Oo_X(1,0,1)$; $c_2(\Ee)=t_1t_2+t_1t_3+t$.
\item $\Oo_X(2,0,0)\oplus \Oo _X(0,1,0)\oplus \Oo_X(0,0,1)$; $c_2(\Ee)=t_2t_3+2t_1t_3+2t_1t_2$ and $c_3(\Ee)=2$.
\item $\Oo_X(1,1,0)\oplus \Oo _X(1,0,0)\oplus \Oo_X(0,0,1)$; $c_2(\Ee)=t_2t_3+2t_1t_3+t_1t_2$ and $c_3(\Ee)=1$.
\item $\Oo_X(1,0,0)^{\oplus 2} \oplus \Oo_X(0,1,1)$; $c_2(\Ee)=t_2t_3+2t_1t_3$ and $c_3(\Ee)=0$.
\item $\Oo_X(1,0,0)^{\oplus 2} \oplus \Oo_X(0,1,0)\oplus \Oo_X(0,0,1)$; $c_2(\Ee)=t_2t_3+2t_1t_3+2t_1t_2$ and $c_3(\Ee)=2$.
\end{enumerate}
In each case of $r=2$ except (1), the associated curve $C$ is a connected, normal and rational curve of multidegree $(e_1, e_2, e_3)$ with $c_2(\Ee)=e_1t_2t_3+e_2t_1t_3+e_3t_1t_2$. Note that
$$\dim \Ext^1(\Oo_X(2,0,0), \Oo_X(0,1,1))=h^1(\Oo_X(-2,1,1))=4.$$
So there are non-trivial extensions, which cannot be decomposable. In the other cases of $r=2$, such extensions are always trivial.
\end{remark}

\begin{example}\label{s1}
We have $\dim \Ext^1 (\Oo_X(2,0,0),\Oo _X(0,1,1)) =4$ and so we have a family $\{\Ee _\lambda \}$ of extensions of $\Oo _X(2,0,0)$ by $\Oo _X(0,1,1)$ with $\lambda\in \Ext^1 (\Oo_X(2,0,0), \Oo_X(0,1,1))$. Each extension shares the same Chern number and the same number of linearly independent sections. Any $\Ee_\lambda$ is isomorphic to $\Uu(0,0,-1)$, where $\Uu$ is an Ulrich bundle arising from an extension of $\Oo _X(2,0,1)$ by $\Oo _X(0,1,2)$ (see \cite{CFM0} Section 7). So if $\Ee=\Ee_{\lambda}$ with $\lambda\ne 0$, then $\Ee$ is indecomposable.

\quad {\emph {Claim 1}}: For $\Ee =\Ee_{\lambda}$ with $\lambda \ne 0$, we have $h^1(\Ee (t,t,t)) =0$ for all $t\in \ZZ$.

\quad {\emph {Proof of Claim 1}}: For every $t\in \ZZ$ we have $h^1(\Oo _X(2+t,t,t)) = h^1(\Oo _X(t,t+1,t+1)) =0$ and so $h^1(\Ee (t,t,t)) =0$.  \qed

\quad {\emph {Claim 2}}: For $\Ee =\Ee_{\lambda}$ with $\lambda \ne 0$, we have $h^2(\Ee (-2,-2,-2)) =1$ and $h^2(\Ee (t,t,t)) =0$ for all $t\ne -2$.

\quad {\emph {Proof of Claim 2}}: For every $t\ne -2$, we have $h^2(\Oo _X(2+t,t,t)) = h^2(\Oo _X(t,t+1,t+1)) =0$ and so $h^2(\Ee (t,t,t)) =0$. For $t=-2$, we have $h^2(\Oo _X(-2,-1,-1)) =0$, $h^3(\Oo _X(-2,-1,-1)) =0$ and $h^2(\Oo _X(0,-2,-2)) =1$. Hence we have $h^2(\Ee (-2,-2,-2))=1$. \qed

Fix any non-trivial extension $\Ee_{\lambda}$ of $\Oo _X(2,0,0)$ by $\Oo _X(0,1,1)$ and then $\Ee _\lambda$ is semistable, but not stable, with respect to the polarization $H=\Oo_X(1,1,1)$. Moreover the line bundles $\Oo _X(2,0,0)$ and $\Oo _X(0,1,1)$ appear in any H\"{o}lder-Schreier (or Harder-Narasimhan) decomposition of $\Ee _\lambda$ with respect to $H$-stable sheaves.

We also see that each automorphism of $\Ee _\lambda$ respect the extension defining $\Ee _\lambda$ and it implies that $\Ee _\lambda$ is simple. This is true also for the following reason: Since $h^0(\Ee _\lambda (-2,0,0)) = h^0(\Ee _\lambda (0,-2,0))
= h^0(\Ee _\lambda (0,0,-2)) =0$ and
\begin{align*}
\Oo _C(2,0,0)\cdot \Oo _X(2,1,1)\cdot \Oo _X(2,&1,1) = 4\\
 &< 8 = \Oo _X(0,1,1)\cdot \Oo _X(2,1,1) \cdot \Oo _X(2,1,1),
\end{align*}
$\Ee _\lambda$ is stable with respect to the polarization $\Oo _X(2,1,1)$.
\end{example}

\begin{lemma}\label{ss1}
Let $\Ee$ be a globally generated vector bundle of rank $r\ge 2$ on $X$ with $c_1=(2,1,1)$ and no trivial factor. Then $c_2(\Ee)=(0,2,2)$ if and only if we have either
\begin{itemize}
\item[(i)] $\Ee\cong \Ee_{\lambda}$ as in Example \ref{s1} or
\item[(ii)] $\Ee \cong \Oo_X(1,0,0)^{\oplus 2} \oplus \Oo_X(0,1,1)$.
\end{itemize}
\end{lemma}
\begin{proof}
Since the ``~if~" part is obvious, so we only need to check the ``~only if~" part. Take $C$ associated to $\Ee$. Since $C$ has multidegree $(0,2,2)$, for each connected component $C_i$ of $C$, there are $o_i\in \PP^1$ and $C'_i\in |\Oo _{\PP^1\times \PP^1}(e[i]_3,e[i]_2)|$ such that
$C_i = \{o_i\}\times C'_i$. Since $\Ii_C(2,1,1)$ is spanned, we get $e[i]_2\le 1$ and $e[i]_3\le 1$ for all $i$. Since $\omega _C(0,1,1)$ is spanned,
no connected component of $C$ is a line. Hence $s=2$ and each $C_i$ is smooth and rational with $(e[i]_1, e[i]_2, e[i]_3)=(0,1,1)$ for $i=1,2$. In particular we have $\omega_C(0,1,1)\cong \Oo_C$ and so we have $r\le 3$.

Assume $r=2$. Since $h^0(\Oo _X(2,0,0)) =3 > 2 =h^0(\Oo _C(2,0,0))$, we get the existence of a non-zero map $f: \Oo  _X(0,1,1)\to \Ee$. Since $h^0(\Ii _C(1,0,0)) = h^0(\Ii _C(2,-1,0)) =h^0(\Ii_C(2,0,-1)) =0$, so $f$ induces an exact sequence
$$0\to \Oo _X(0,1,1) \to \Ee \to \Ii _T(2,0,0)\to 0$$
with $T$ a locally complete intersection $0$-dimensional subscheme of codimension $2$. Since $c_2(\Ee )=2t_1t_3+2t_1t_2$, we get $T= \emptyset$ and so $\Ee \cong \Ee_{\lambda}$ as in Example \ref{s1}.

Assume $r=3$. We have $h^0(\Ee )=2+h^0(\Ii _C(2,1,1)) = 8$. As in the case $r=2$, we get an injective map $f: \Oo_X(0,1,1) \to \Ee$ with globally generated cokernel $\Ff:=\mathrm{coker}(f)$. From $h^1(\Oo _X(0,1,1)) =0$, we get $h^0(\Ff )=4$. Note that $h^0(\Ee(-1,-1,-1))=h^0(\Ee(0,-2,-1))=h^0(\Ee(0,-1,-2))=0$ and so $\Ff$ is torsion-free. Since $h^0(\Ii _C(1,1,1)) = 2$ and $h^1(\Oo _X(-1,0,0)) =0$, we first get
an injective map $\Oo _X(1,0,0)^{\oplus 2} \to \Ee$ and then a non-zero map $w: \Oo _X(1,0,0)^{\oplus 2} \to \Ff$. First assume that $w$ is injective. Since
$c_1(\Ff )=(2,0,0)$, $w$ is an isomorphism (even if a priori $\Ff$ is only torsion-free). Since $h^1(\Oo_X(-1,1,1))=0$, so the extension induced by $f$ is trivial and $\Ee \cong \Oo_X(1,0,0)^{\oplus 2} \oplus \Oo_X(0,1,1)$. Now assume that $w$ is not injective. We claim that $\mathrm{Im}(w) =\Oo _X(2,0,0)$. Indeed, $\mathrm{Im}(w)$ is a spanned torsion-free sheaf of rank $1$ with $c_1=(x,b,c)$, $b\ge 0$, $c\ge 0$. Since $\mathrm{Im}(w) \subset \Ff$ and $\Ff/\mathrm{Im}(w)$ is spanned, we have
$x\le 2$, $b\le 0$ and $c\le 0$. Since $\mathrm{Im}(w) \ne \Oo _X(1,0,0)$, we get
$x=2$ and $b=c=0$. Since $\Ff$ is spanned, $h^0(\Ff )=4$, $\Ff/\mathrm{Im}(w)$ is spanned
by the cokernel of the injection $H^0(\mathrm{Im}(w)) \to H^0(\Ff)$ and $\Ff$ has no trivial factor, we get
a contradiction.
\end{proof}

\begin{remark}
For any non-trivial extension $\Ee_{\lambda}$ in Example \ref{s1}, we can compute $h^1(\Ee_{\lambda}^\vee)=1$ and so we have a unique non-trivial extension $\Gg_{\lambda}$ of $\Ee_{\lambda}$ by $\Oo_X$. By Lemma \ref{ss1} we get
$\Gg\cong\Oo_X(1,0,0)^{\oplus 2} \oplus \Oo_X(0,1,1)$ for any $\lambda \ne 0$.
\end{remark}

\begin{proposition}\label{sss2}
Let $\Ee$ be a globally generated vector bundle of rank $r\ge 2$ on $X$ with $c_1(\Ee)=(2,1,1)$, multidegree $(1,2,0)$ and no trivial factor. Then we have
$$\Ee \cong \Oo_X(2,1,0)\oplus \Oo_X(0,0,1).$$
\end{proposition}
\begin{proof}
Since no connected component of the associated curve $C$ is a line and $\deg ({C}) =3$, then $s=1$.
Since $\Ii _C(2,2,2)$ is spanned, $C$ is not a plane cubic.
Hence $C$ is a connected and rational curve of degree $3$. Since $\omega_C(0,1,1) \cong \Oo_C$ and $s=1$, we have $r=2$ (see Remark \ref{24nov}). Since $h^0(\Oo_X(0,0,1))=2=h^0(\Oo_C(0,0,1))+1$, so we have $h^0(\Ee(-2,-1))=1$. Since $h^0(\Ii_C(-1,0,1))=h^0(\Ii_C(0,-1,1))=h^0(\Ii_C))=0$, so we get an exact sequence
$$0\to \Oo_X(2,1,0) \to \Ee \to \Ii_T(0,0,1)\to 0$$
with $T$ a locally complete intersection $0$-dimensional subscheme of codimension $2$. Since $c_2(\Oo_X(2,1,0)\oplus \Oo_X(0,0,1))=(1,2,0)$, so we have $T=\emptyset$. Since $h^1(\Oo_X(2,1,-1))=0$, so the extension is trivial.
\end{proof}

\begin{lemma}\label{sss2}
Let $\Ee$ be a globally generated vector bundle of rank $r\ge 2$ on $X$ with no trivial factor,
$c_1(\Ee)=(2,1,1)$ and multidegree $(0,e_2,e_3)$ with $e_2+e_3\le 2$. Then we have $\Ee \cong \Oo_X(1,1,1)\oplus \Oo_X(1,0,0)$.
\end{lemma}
\begin{proof}
Since $\deg ({C}) = e_1+e_2+e_3 \le 2$ and no connected component of $C$ is a line,
then $s=1$ and $C$ is a smooth conic. Since $\omega_C(0,1,1)\cong \Oo_C$ and $s=1$, then $r=2$ (Remark \ref{24nov}).

Fix any smooth conic $C$ with $(e_1, e_2, e_3)=(0,1,1)$. There are $p\in \PP^1$ and $C'\in |\Oo _{\PP^1\times \PP^1}(1,1)|$ such that $C = \{p\}\times C'$. Since $\Ii _{C',\PP^\times \PP^1}(1,1)$ is globally generated, so is $\Ii _C(2,1,1)$. Since $\Oo _C(1,0,0) \cong \Oo _C$ and $h^0(\Oo _X(1,0,0)) =2$, we have $h^0(\Ee(-1,-1,-1))=h^0(\Ii _C(1,0,0)) >0$ and so we can pick a non-zero map $f: \Oo _X(1,1,1) \to \Ee$. Since $h^0(\Ii _C(0,0,0)) = h^0(\Ii _C(1,-1,0)) = h^0(\Ii _C(1,0,-1))=0$, we see that $f$ induces an exact sequence
$$0\to \Oo _X(1,1,1) \to \Ee \to \Ii _T(1,0,0)\to 0$$ with $T$ a locally complete intersection curve. Since $(t_1+t_2+t_3)t_1 =t_1t_2+t_1t_3$ and $c_2(\Ee ) =(0,1,1)$,
we get $T=\emptyset$. Since $h^1(\Oo _X(0,1,1)) =0$, we get $\Ee \cong \Oo _X(1,1,1)\oplus \Oo _X(1,0,0)$.
\end{proof}

\begin{lemma}\label{sss3}
Let $\Ee$ be a globally generated vector bundle of rank $r\ge 2$ on $X$ with $c_1(\Ee)=(2,1,1)$ and no trivial factor. If the associated curve $C$ is a curve of multidegree $(e_1, e_2, e_3)=(1,1,1)$, then we have $\Ee \cong \Oo_X(1,1,0)\oplus \Oo_X(1,0,1)$.
\end{lemma}
\begin{proof}
Lemma \ref{salto1} gives $s=1$. Since $\omega_C(0,1,1)$ is trivial and $s=1$, so we have $r=2$. Since $\deg (\Oo_C(1,1,0))+1=4=h^0(\Oo_X(1,1,0)$, so we have $h^0(\Ee(-1,0,-1))=h^0(\Ii_C(1,1,0))=1$. Thus $\Ee$ fits in an exact sequence
$$0\to \Oo_X(1,0,1) \to \Ee \to \Ii_T(1,1,0) \to 0$$
with either $T=\emptyset$ or $T$ a locally complete intersection curve. Since $(t_1+t_2)(t_1+t_3)=t_1t_2+t_2t_3+t_3t_1 =c_2(\Ee) $, so we have $T=\emptyset$. The vanishing of $H^1(\Oo_X(0,-1,1))$ implies that the extension is trivial.
\end{proof}

\begin{example}\label{s3}
Let $\Gg=\pi_{23}^*(\phi_p^*T\PP^2(-1))$, where $\phi_p : Q \to \PP^2$ is the linear projection with the center $p\in \PP^3 \setminus Q$. We have $\dim \Ext^1 (\Oo_X(2,0,0),\Gg)= h^0(\Gg) =3$ and so we have a family $\{\Ee _\lambda \}$ of non-trivial extensions of $\Oo _X(2,0,0)$ by $\Gg$ with
$c_1=(2,1,1)$ and $c_2=(0,3,3)$. The bundles $\Ee_\lambda$ do not split because $h^0(\Ee_\lambda)=6$ and there are no bundles of rank $3$ in the list of Remark \ref{allspl} with $6$ sections. Lemma \ref{salto1} gives $s=3$ and hence $h^0(\Ee_{\lambda}^\vee)=1$ and so we have a unique non-trivial extension $\Ff$ between $\Oo_X$ and $\Ee_{\lambda}$. Taking the counter-image $\Aa$ by the subbundle $\Oo _X(2,0,0)$ of $\Gg$, we get that $\Ff$ is an extension of $\Gg$ by $\Aa$. Since $\Aa$ is an extension of $\Oo _X(2,0,0)$ by $\Oo_X$, we have either $\Aa \cong \Oo _X(1,0,0)^{\oplus 2}$ or $\Aa \cong \Oo _X\oplus \Oo _X(2,0,0)$. In the first case we get $\Ff \cong \Oo _X(1,0,0)^{\oplus 2}\oplus \Gg$, because
$h^1(\Gg ^\vee (-1,0,0)) =0$ by K\"{u}nneth. In the second case we get that $\Oo _X$ is a factor of $\Ff$, because $h^1(\Gg ^\vee )=0$ by K\"{u}nneth.
\end{example}

\begin{example}\label{s2}
Since $h^1(\Oo _X(-2,1,0)) =h^1(\Oo _X(-2,0,1)) =2$, there are non-trivial extensions
\begin{equation}\label{equ2+}
0 \to \Oo _X(0,0,1)\oplus \Oo _X(0,1,0) \to \Ee \to \Oo _X(2,0,0)\to 0
\end{equation}
and $\Ext^1$ is a $4$-dimensional vector space. In this vector space the origin corresponds to $\Oo _X(2,0,0)\oplus \Oo _X(0,1,0) \oplus \Oo _X(0,0,1)$, while two $2$-dimensional linear spaces correspond to extensions $\Oo _X(0,0,1)\oplus \Ff$ and $\Oo _X(0,1,0)\oplus \Gg$ with
$\Ff = \pi _{12}^\ast (\Ff ')$, $\Gg =\pi _{13}^{\ast}(\Gg ')$ where $\Ff'$, $\Gg'$ are spanned on $Q=\PP^1\times \PP^1$ and of type $(2,1)$. By \cite[Proposition 3.7]{BHM2} such bundles $\Ff'$ and $\Gg'$ are in the following list:
\begin{enumerate}
\item $ 0\to \Oo_Q \to \Oo_Q(0,1)\oplus \Oo_Q(1,0)^{\oplus 2} \to \Hh \to 0$ ; $c_2(\Hh)=2$
\item $0\to \Oo_Q(-1,-1) \to \Oo_Q(1,0)\oplus \Oo_Q^{\oplus 2} \to \Hh \to 0$ ; $c_2(\Hh)=3$.
\item $0\to \Oo_Q(-2,-1) \to \Oo_Q^{\oplus 3} \to \Hh \to 0$ ; $c_2(\Hh)=4$.
\end{enumerate}

Every bundle $\Ee$ in (\ref{equ2+}) is globally generated with $h^0(\Ee )=7$, $h^1(\Ee )=0$ and $c_2(\Ee)=(e_1, e_2, e_3)=(1,2,2)$. Lemma
\ref{salto1} gives $s=1$. Hence $C$ is a smooth rational curve.
Let $(\epsilon_1,\epsilon_2)\in H^1(\Oo _X(-2,1,0))\times H^1(\Oo _X(-2,0,1))$ denote the extension class representing $\Ee$.

\quad {\emph {Claim }}: If $\epsilon_1\ne 0$ and $\epsilon_2\ne 0$, then $\Ee$ is indecomposable.

\quad {\emph {Proof of Claim }}: $\Oo _X(2,0,0)$ cannot be a factor of $\Ee$, because a non-zero map $\Oo _X(2,0,0) \to \Ee$ gives a splitting of (\ref{equ2}). Similarly no $\Oo _X(a,b,c)$ with $a+b+c \ge 2$ may be a direct factor of $\Ee$. Since $\Ee$ is globally generated, no $\Oo _X(a_1,a_2,a_3)$ with $a_1<0$ is a direct factor of $\Ee$. To prove the Claim it is sufficient to prove that neither  $\Oo _X(1,0,0)$ nor $ \Oo _X(0,1,0)$ is a factor of $\Ee$. Assume for instance that $\Oo _X(1,0,0)$ is a factor and call $f: \Ee \to \Oo _X(0,0,1)$ the associated surjection. Since $h^0(\Oo _X(1,-1,0))=0$, we have $f_{\vert{\Oo _X(0,1,0)}} \equiv 0$ and so the associated map
$f_{\vert{\Oo _X(0,0,1)\oplus \Oo _X(0,1,0)}}$ is either zero or a surjection. In the former case we get a surjection $\Oo _X(2,0,0)\to \Oo _X(0,0,1)$, absurd. In the latter case we get $\epsilon_2=0$. \qed
\end{example}

\begin{proposition}\label{kk1}
There are globally generated vector bundles $\Ee$ of rank $r$ with the Chern classes $c_1=(2,1,1)$, $c_2=(1,2,2)$ and no trivial factor if and only if $3\le r\le 4$.
\begin{enumerate}
\item For each $r\in \{3,4\}$, the family of such bundles is parametrized by an irreducible family.
\item For any such a bundle,  the associated curve $C$ is connected and we have $h^0(\Ee (-1,0,0)) >0$, $h^1(\Ee )=0$.
\item \begin{itemize}
\item[(i)] If $r=3$, then $\Ee$ is as in Example \ref{s2}.
\item[(ii)] If $r=4$, then $\Ee \cong \Oo_X(1,0,0)^{\oplus 2} \oplus \Oo_X(0,1,0)\oplus \Oo_X(0,0,1)$.
\end{itemize}
\end{enumerate}
\end{proposition}

\begin{proof}
\quad{(a)} Lemma \ref{salto1} gives that $s=1$. Since $e_1=1$, $C$ is rational. We also get $\deg (\omega _C(0,1,1)) =2$ and so $h^0(\omega _C(0,1,1))=3$. Therefore for a fixed $C$ we get bundles with no trivial factor
if and only if $3\le r \le 4$ (see Remark \ref{24nov}).

\quad{(b)} Let $\Theta$ be the set of all smooth rational curve $E\subset X$ with multidegree $(1,2,2)$.
The set $\Theta$ is an irreducible variety. Any $C$ coming from a bundle $\Ee$ as in Example \ref{s2} is an element of $\Theta$. Since $h^2(\Oo _X)=0$ and $h^1(\Ee ) =0$ for such bundle, we get $h^1(\Ii _C(2,1,1)) =0$. By semicontinuity the same is true for all $E\in \Theta '$ with $\Theta '$ a non-empty open subset of $\Theta$. Since $\Ii _C(2,1,1)$ is globally generated and this condition gives an open subset for families of ideal sheaves of smooth curves with constant $h^0$, there is a non-empty open subset $\Theta ''$ of $\Theta '$ such that each $E\in \Theta''$ gives a globally generated vector bundle $\Ee$ with $E$ as one of its zero-loci and
with $h^1(\Ee )=0$. Assume $h^1(\Ee )>0$, i.e. assume $h^1(\Ii _C(2,1,1)) >0$. Since the multiplication map $H^0(\Oo _C(1,0,0))\otimes H^0(\Oo _C(1,1,1)) \to H^0(\Oo _C(2,1,1))$ is surjective, we get $h^1(\Ii _C(1,1,1)) >0$, i.e. the smooth rational curve $C\subset X\subset \PP^7$ of degree $5$ is not linearly normal. Thus its linear span $\langle C\rangle$ has dimension at most $4$. The formula for the number of trisecant lines of smooth curves in $\PP^4$ due to Castelnuovo and Berzolari \cite[Section 1.1]{l}
gives the existence of a line $D\subset \langle C\rangle$ such that $\deg (D\cap C)\ge 3$. Since $X$ is cut out by quadrics in $\PP^7$, we get $D\subset X$. Hence $D$ is in the base locus of $\Ii _C(2,1,1)$, a contradiction.

\quad{(c)} Take any $r\in \{3,4\}$. Fix $C\in \Theta''$ and then we have $h^0(\Oo _C(2,1,0)) = 5 =h^0(\Oo _X(2,1,0)) -1$. Since $h^1(\Oo _X(0,0,-1)) =0$, we get $h^0(\Ee (0,0,-1)) =1$ and so, up to a scalar, there is a unique injective map $f: \Oo _X(0,0,1) \to \Ee$ with cokernel $\Ff := \mathrm{coker} (f)$. Since $h^0(\Ee (-1,0,-1)) = h^0(\Ee (0,-1,-1))= h^0(\Ee (0,0,-2)) =0$, so $\mathrm{Im}(f)$ is saturated in $\Ee$, i.e. $\Ff$ is torsion-free.

\quad {\emph {Claim }}: $\Ff$ is locally free.

\quad {\emph {Proof of Claim }}: $\Ff$ is reflexive if and only if $h^1(\Ff (-t)) =0$ for $t\gg 0$  by \cite[Remark 2.5.1]{Hartshorne1}. This is true, because $h^1(\Ee (-t))=0$ if $t\gg 0$ and $h^2(\Oo _X(-t,-t,-t-1)) =0$ if $t \ge 0$. Assume first $r=3$. Since $\Ff$ is a reflexive sheaf of rank $2$, so $\Ff$ is locally free if and only if $c_3(\Ff )=0$ by \cite[Proposition 2.6]{Hartshorne1}. The integer $c_3(\Ff)$ depends only on the Chern classes of $\Ee$ and of $\Oo _X(0,0,1)$ and so it may computed taking as $\Ee$ any of the bundles in Example \ref{s2}. For these $\Ee$ we get $c_3(\Ff ) = c_3(\Oo _X(0,1,0)\oplus \Oo _X(0,0,2)) =0$. Now assume $r=4$. We take as $\Ee '$ the quotient of $\Ee$ by the image of a general section of $\Ee$. We call $\Ff'$ the sheaf $\Ee '/\Oo _X(0,0,1)$. We just proved that $\Ff'$ is locally free. Since $\Ff$ is an extension of $\Ff '$ by $\Oo _X$, it is locally free, concluding the proof of the \emph{Claim}. \qed

Similarly as above we have an injective map $f:\Oo_X (0,1,0) \to \Ff$ whose cokernel $\Gg :=\mathrm{coker}(f)$ is locally free. It gives us an exact sequence
$$0\to \Oo_X(0,1,0)\oplus \Oo_X(0,0,1) \to \Ee \to \Gg \to 0.$$
If $r=3$, we have $\Gg\cong \Oo_X(2,0,0)$ and so we are in sequence (\ref{equ2+}). If $r=4$, then we have $\Gg \cong \Oo_X(1,0,0)^{\oplus 2}$. Since $h^1(\Oo_X(-1,1,0))=h^1(\Oo_X(-1,0,1))=0$, so the extension is trivial.
\end{proof}

\begin{remark}
Take $C$ as in the proof of Proposition \ref{kk1}, i.e. $C\subset X$ is a smooth and connected rational curve with multidegree $(1,2,2)$ such that $h^1(\Ii _C(2,1,1))=0$, i.e. $h^0(\Ii _C(2,1,1)) =9$, and $\Ii _C(2,1,1)$ globally generated. Since $h^1(C,TX_{\vert_C}) =0$, so the normal bundle $N_{C|X}$ of $C$ in $X$ satisfies $h^1(C,N_{C|X}) =0$. The vector bundle $N_{C|X}$ has rank two and degree
$$\deg (TX_{\vert_C}) -2 = 2\times 1+2\times 2+2\times 2 -2=8.$$
Thus we have $h^0(C,N_{C|X}) = 10$ and so the Hilbert scheme of $X$ is smooth at $C$ and its unique irreducible component, say $\mathbb {H}$, containing $C$ has dimension $10$. Let $\mathbb {H}'$ be the non-empty open subset of $\mathbb {H}$ parametrizing all smooth and connected rational curves $C$ with the numerical invariants above, i.e. $\deg (\Oo _C(1,0,0)) =1$, $\deg (\Oo _C(0,1,1))= \deg (\Oo _C(0,0,1)) =2$, $h^1(\Ii _C(2,2,1))=0$ and $\Ii _C(2,1,1)$ globally generated. For any $C\in \mathbb {H}'$ we have $h^0(\omega _C(0,1,1)) = 3$.
If $r=4$, then the Hartshorne-Serre correspondence shows that for each $C\in \mathbb {H}'$ we associate a unique globally generated $\Ee$ with $C$ as the zero-locus of a section and $h^1(\Ee )=0$. Since $r=4$, we get $h^0(\Ee ) =12$ and so $\PP H^0(\Ee ) \cong \PP^{11}$.
\end{remark}

\begin{example}\label{s4}
Let $\Gg:=\phi^*T\PP^3(-1)$, where $\phi : X \to \PP^3$ is the linear projection.
We have $\dim \Ext^1 (\Gg, \Oo(1,0,0))= h^1(\Gg^\vee(1,0,0)) \geq 4$ and so we have a family $\{\Ff _\lambda \}$ of non-trivial extensions of  $\Gg$ by $\Oo(1,0,0)$ with
$c_1=(2,1,1)$ and $c_2=(2,3,3)$ with the exact sequence
\begin{equation}\label{eqqa}
0\to \Oo_X^{\oplus 3} \to \Ff_{\lambda} \to \Ii_C(2,1,1)\to 0,
\end{equation}
where $C$ is a smooth curve of multidegree $(2,3,3)$. From the proof (c2) of Theorem \ref{prop4.2}, we get $p_a(C)=2$ and in particular $C$ is hyperelliptic. Thus $\Oo_C(1,0,0)$ is the unique $g_2^1$ and so canonical. Since $\omega_C(-1,0,0)\cong \Oo_C$, so the Hartshorne-Serre correspondence implies the existence of a globally generated vector bundle $\Hh$ fitting into the sequence
$$0\to \Oo_X \to \Hh \to \Ii_C(3,2,2) \to 0.$$
Note that $h^0(\Hh(-2,-1,-1))=h^0(\Ii_C(1,1,1))=h^0(\Ff_{\lambda}(-1,0,0))=1$. From the sequence (\ref{eqqa}), we also have $h^0(\Hh(-3,-1,-1))=h^0(\Hh(-2,-2,-1))=h^0(\Hh(-2,-1,-2))=0$ and so a non-zero section in $H^0(\Hh(-2,-1,-1))$ induces an exact sequence
$$0\to \Oo_X(2,1,1) \to \Hh \to \Ii_T(1,1,1) \to 0$$
with $T$ a locally complete intersection $0$-dimensional subscheme of codimension $2$. Since $c_2(\Hh)=(2,3,3)$, so we get $T=\emptyset$. Since $h^1(\Oo_X(1,0,0))=0$, so the extension is trivial. Thus we have $\Hh \cong \Oo_X(2,1,1)\oplus \Oo_X(1,1,1)$ and $\Ii_C$ admits the following locally free resolution:
$$0\to\Oo_X(-1,-1,-1)\to\Oo_X\oplus\Oo_X(1,0,0)\to\Ii_C(2,1,1)\to 0.$$
Moreover $h^1(\Ff_{\lambda}^\vee)=4$ so we have higher rank bundles up to $r=8$ with the same Chern classes and no trivial factors.
\end{example}

\begin{theorem}\label{prop4.2}
Let $\Ee$ be a globally generated vector bundle of rank $r\ge 2$ on $X$ with the Chern classes $c_1=(2,1,1)$, $c_2=(e_1, e_2, e_3)$ and no trivial factor. If the associated curve $C$ is not connected with $s\ge 2$ components, then each component of $C$ has the same multidegree and
\begin{enumerate}
\item $(s;e_1, e_2, e_3)=(3;0,3,3)$ ; $2\le r \le 4$,
\item $(s;e_1, e_2, e_3)=(2;0,2,2)$ ; $2\le r \le 3$ ;\begin{itemize}
\item[(a)] $\Ee\cong \Ee_{\lambda}$ as in Example \ref{s1} or
\item[(b)] $\Ee \cong \Oo_X(1,0,0)^{\oplus 2} \oplus \Oo_X(0,1,1)$.
\end{itemize}
\end{enumerate}
If $C$ is connected, then $(p_a(C);e_1, e_2, e_3)$ and the rank $r$ are as follows:
\begin{enumerate}
\item $(0;0,1,1)$ ; $ \Ee \cong \Oo_X(1,1,1)\oplus \Oo_X(1,0,0)$,
\item $(0;2,1,1)$ ; $\Ee\cong \Uu(0,0,-1)$ where $\Uu$ is an Ulrich bundle with $c_1=(2,1,3)$ and $c_2=(3,3,1)$ in \cite[Section 7]{CFM0},
\item $(0;1,a,b)$ with $a+b \ge 2$ ; $2\le r \le a+b$,
\begin{itemize}
\item [(i)] $(a,b)=(1,1) \Leftrightarrow \Ee \cong \Oo_X(1,1,0)\oplus \Oo_X(1,0,1)$.
\item [(ii)] $(a,b)=(2,0) \Leftrightarrow \Ee \cong \Oo_X(2,1,0)\oplus \Oo_X(0,0,1)$.
\item [(iii)] $(a,b)=(2,2) \Leftrightarrow$ \begin{itemize}
\item[(a)] $\Ee$ as in Example \ref{s2} or
\item[(b)] $\Ee \cong \Oo_X(1,0,0)^{\oplus 2} \oplus \Oo_X(0,1,0)\oplus \Oo_X(0,0,1),$
\end{itemize}
\end{itemize}
\item $(1;2,2,2)$ ; $3 \le r \le 5$,
\item $(2;2,3,3)$ ; $3\le r \le 8$,
\item $(3;2,4,4)$ ; $3 \le r \le 11$.
\end{enumerate}
\end{theorem}

\begin{proof}
\quad (a) Assume first that $s\ge 2$. By Lemma \ref{salto1}, we have $2\le s \le 3$ and $(e_1,e_2,e_3) = (0,s,s)$. If $s=2$, then we are as in Example \ref{s1} due to Lemma \ref{ss1}. Now assume $s=3$.

\quad{\emph{Claim }}: $\Ii _C(2,1,1)$ is always globally generated. \\
\quad{\emph{Proof of Claim }}: We have $e[i]_1=0$ and $e[i]_2=e[i]_3=1$ for $i=1,2,3$ (see Lemma \ref{salto1}). Thus there are three distinct points $p_1,p_2,p_3\in \PP^1$ and $C'_i\in |\Oo _{\PP^1\times \PP^1}(1,1)|$ for $i=1,2,3$ such that $C_i = \{p_i\}\times C'_i$. Write $W:= \{p_1,p_2,p_3\}\times \PP^1\times \PP^1$ and then it is the union of three disjoint elements of $|\Oo _X(1,0,0)|$ and so the restriction map $\rho : H^0(\Oo _X(2,1,1))\to H^0(W,\Oo _W(2,1,1))$ is bijective. Since $\Ii _{C'_i,\PP^1\times \PP^1}(1,1) \cong \Oo _{\PP^1\times \PP^1}$ via the automorphism induced by an equation of $C'_i$, so the line bundle $\Ii _{C,W}(2,1,1)$ is globally generated. Hence the scheme-theoretic base locus of $\Ii _C(2,1,1)$ is disjoint from $W$ and in particular it is disjoint from $C$. Since $\Ii _Y(2,1,1)$ is globally generated, the scheme-theoretic base locus of $\Ii _C(2,1,1)$ is contained in $Y=C\cup D$. Hence it is contained in $D\setminus (C\cap D)$. Assume the existence of $p\in D\setminus (C\cap D)$ with $p$ in the base locus. Since $Y \cong \pi _{12}(Y) \in |\Oo _{\PP^1\times \PP^1}(4,2)|$ and $\pi _{12}(D) \in |\Oo _{\PP^1\times \PP^1}(1,2)|$, so we have $\deg (D\cap C) =6$. Since $\deg (D\cap W) = 3 \cdot \deg (\Oo _D(0,1,1)) =6$, we get $2= h^0(\Ii _{C\cup D}(2,1,1)) = h^0(\Ii _{C\cup \{p\}}(2,1,1)) =h^0(\Ii _C(2,1,1))$. Since $h^0(\Oo _C(2,1,1)) = 9$, we get $h^0(\Ii _C(2,1,1)) >2$, a contradiction. \qed

Now let us assume that $C$ is connected, i.e. $s=1$. Since $Y$ has multidegree $(2,4,4)$, so we have $e_1\le 2$ and $e_2, e_3\le 4$.

\quad (b) Assume that $C$ is rational. In particular we have $\deg (\omega _C) = -2$ and $\deg (\omega _C(0,1,1)) = -2+e_2+e_3$. Thus we have $\omega _C(0,1,1) \cong \Oo _C$ if and only if $e_2+e_3=2$ and $\omega _C(0,1,1)$ is spanned if and only if $e_2+e_3 \ge 2$. Since ${\pi _{12}}_{\vert_C}$ is an embedding, we have either $e_1=1$ or $e_2=1$.  Again since ${\pi _{13}}_{\vert_C}$ is an embedding, we have either $e_1=1$ or $e_3=1$. Thus if $e_1\ne 1$, then we have $e_2=e_3=1$. Note that any case, if it exists, must have $e_2+e_3 =2$ and so $\omega _C \cong \Oo _C(0,-1,-1)$; so these cases cannot occur for $r>2$ with no trivial factor. Since
the intersection $Y$ of two general elements of $|\Ii _C(2,1,1)|$ has multidegree $(2,4,4)$, we have $e_1\le 2$. If $e_1=0$, we have $\Ee \cong \Oo_X(1,1,1)\oplus \Oo_X(1,0,0)$ by Lemma \ref{sss2}. If $e_1=1$, we have $\Ee \cong \Oo_X(1,1,0)\oplus \Oo_X(1,0,1)$ by Lemma \ref{sss3}. If $e_1=2$, the bundle $\Ee$ exists, because $\Ii _C(2,1,1)$ is spanned by step (c) of the proof of Proposition \ref{a2.1++}, which is the case $\mathrm{(iii)}$ of Proposition \ref{a2.1++} with $(0;2,1,1)$. If $\Ee$ is a globally generated vector bundle of rank $2$ with $c_1=(2,1,1)$ and $c_2=(2,1,1)$, we have the exact sequence
\begin{equation}\label{Ext2}
0\rightarrow \Oo_X\rightarrow \Ee\rightarrow \Ii_{Y}(2,1,1)\rightarrow 0,
\end{equation} 
where $Y$ is a smooth curve of multidegree $(2,1,1)$. Then $\Ee(0,0,1)$ is a globally generated vector bundle of rank $2$ with $c_1=(2,1,3)$ and $c_2=(3,3,1)$ and the zero locus of a general section is a smooth curve $C$ of multidegree $(3,3,1)$ with the exact sequence 
$$0\rightarrow \Oo_X\rightarrow \Ee(0,0,1)\rightarrow \Ii_{C}(2,1,3)\rightarrow 0.$$ 
From its twist by $(-1,0,-2)$, we get $h^0(\Ee(-1,0,-1))= h^0(\Ii_{C}(1,1,1))$. But from (\ref{Ext2}) twisted by $(-1,0,-1)$, we get $h^0(\Ee(-1,0,-1))= h^0(\Ii_{Y}(1,1,0))=0$. It implies that $C$ is non-degenerate and so $\Ee(0,0,1)$ is an Ulrich bundle by \cite[Lemma 7.2]{CFM0}.

\quad (c) Assume that $C$ has positive genus $g>0$ and hence $e_1\ne 1$. If $Y$ is the complete intersection of two general elements of $|\Ii _C(2,1,1)|$, then we have $p_a(Y) = 3$ and so $g \le 3$. Since ${\pi _{12}}_{\vert_C}$ and ${\pi _{13}}_{\vert_C}$ are embeddings and $e_1\le 2$, so we have
$$(e_1, e_2, e_3)=
                                  \left \{         \begin{array}{llll}
                                             (2,2,2), & \hbox{if $g=1$;}\\
                                             (2,3,3), & \hbox{if $g=2$;} \\
                                             (2,4,4), & \hbox{if $g=3$.}
                                             \end{array}\right.
                                         $$
Thus the case $g=3$ occurs only when $C=Y$, i.e. $\Ee$ is as in Remark \ref{ooo0}.

\quad{(c1)} If $g=1$, then we have $(e_1, e_2, e_3)=(2,2,2)$. Since $h^0(\Oo _C(1,1,1)) =6$, so $C$ is contained a linear subspace $V\subset \PP^7$ of codimension $2$. First assume that $\dim (V\cap X)=1$. Since $\deg (X) =6= \deg ( C)$, we get $C = X\cap V$ as schemes. Since $\Ii _C(1,1,1)$ is spanned, $\Ii _C(2,1,1)$ is spanned. In this case we also get $h^0(\Ee (-1,0,0)) =2$. Now assume $\dim (X\cap V) =2$. Take any surface $S\subseteq X\cap V$ with $S\in |\Oo _X(b_1,b_2,b_3)|$. Since $S\subset V$ and $V\subsetneq \PP^7$, there is $i\in \{1,2,3\}$ with $b_i =0$. Since $e_i >0$, we have $C\nsubseteq S$. Since $\deg (S) +\deg ({C}) > 6 =\deg (X)$, \cite[Theorem 2.2.5]{fov} gives a contradiction. Since $\deg (\Oo_C(0,1,1))=4$, so we have $h^0(\omega_C(0,1,1))=4$.

\quad{(c2)} Now assume $g=2$ and then we have $(e_1, e_2, e_3)=(2,3,3)$. Since $\deg ({C})=8$, we have $h^0(\Oo _C(1)) = 7$. Thus there is a hyperplane $H\subset \PP ^7$ such that $C\subset X\cap H$. Let $Y\subset X$ be the general intersection of two general elements of $|\Ii _C(2,1,1)|$. By the Bertini theorem $Y$ is smooth outside $C$ and so $Y = C\cup D$ with $\deg (D) =2$ and $D$ reduced. This case cannot occur when $r=2$, because $\omega _C \ne  \Oo _C(0,-1,-1)$ since $g>1$. Since $h^0(\omega _C(0,1,1)) = 7$, the existence of spanned $\Ii_C(2,1,1)$ would imply the existence of a bundle with no trivial factor if and only if $3 \le r\le 8$. Since $Y$ has multidegree $(2,4,4)$, $D$
has multidegree $(0,1,1)$.

To prove that this case gives an example we reverse the construction. We start with a smooth rational curve $D\subset X$ with multidegree $(0,1,1)$. There are a point $o\in \PP^1$ and a smooth conic $E\in |\Oo _{\PP^1\times \PP^1}(1,1)|$ such that $D = \{o\}\times E$. Let $Y$ be the intersection of two general elements of $|\Ii _D(2,1,1)|$. By the Bertini theorem $Y$ is smooth outside $D$. It is easy to check that $D$ appears with multiplicity one in $Y$ and thus we have $Y =D\cup C$ with $C$ a reduced curve with $(e_1, e_2, e_3)=(2,3,3)$. Assume for the moment that $C$ is smooth, connected and of genus $2$. Since $Y$ is in the intersection of two elements of $|\Oo _X(2,1,1)|$, so $\Ii _C(2,1,1)$ is globally generated, except at most at the points of $Y$. Let $Y'\subset X$ be the intersection of two general elements of $|\Ii _C(2,1,1)|$. We saw that $Y' = C\cup D'$ as schemes with $\deg (\Oo _{D'}(1,0,0)) =0$ and $\deg (\Oo _{D'}(0,1,0)) =\deg (\Oo _{D'}(0,0,1))=1$, i.e. there are $o'\in \PP^1$ and $E'\in |\Oo _{\PP^1\times \PP^1}(1,1)|$ such that $D' = \{o'\}\times E'$. If $o'\ne o$, then $D\cap D' =\emptyset$ and so $\Ii _C(2,1,1)$ is globally generated. Now assume $o' =o$ for a general $Y'$. Since $h^0(\Oo _C(2,1,1)) = 10+1-2=9$ by the Riemann-Roch theorem, we have $h^0(\Ii _C(2,1,1)) \ge 3$. We have $\deg (C\cdot \{o\}\times \PP^1\times \PP^1) = e_1 =2$ and so $\deg (D\cap C) \le 2$ and $\deg (D'\cap C) \le 2$. Since $p_a(Y)=p_a(Y') = 3$ and $D, D'$ are smooth and rational, we get $\deg (D'\cap C) =\deg (D\cap C)=2$. Set $Z:= C\cap \{o\} \times \PP^1\times \PP^1$. Since $D'\cap C\subseteq Z$ and $D\cap C\subseteq Z$ and $\deg (Z)=2$, we get $D\cap C =D'\cap C =Z$. Write $Z= \{o\}\times Z'$ with $Z'\subset \PP^1\times \PP^1$ and $\deg (Z')=2$. Since $h^0(\PP^1\times \PP^1,\Ii _{Z'}(1,1)) =2$, we get $h^0(\Ii _C(2,1,1)) \le 2$, a contradiction. Since $h^0(\omega _C(0,1,1)) = 7$, this case gives rank $r$ bundles
if and only if $3\le r \le 8$.
\end{proof}


\section{Case of $c_1=(2,2,1)$}
Let $\Ee$ be a globally generated vector bundle of rank $2$ on $X$ with $c_1=(a_1,a_2,a_3)$ with $0<a_i\le 2$ for all $i$ and $a_1+a_2+a_3=5$. Without loss of generality we may assume $c_1=(2,2,1)$. Any complete intersection $Y$ of two elements of $|\Oo _X(2,2,1)|$ has multidegree $(4,4,8)$, $\omega _Y\cong \Oo _Y(2,2,0)$ and hence $p_a(Y) =9$. We always take as $Y$ the complete intersection of two general elements of $|\Ii _C(2,2,1)|$ and write $Y = C\cup D$ with $D$ of multidegree $(4-e_1,4-e_2,8-e_3)$. Since $\pi_{{12}_{|C}}$ is an embedding by Lemma \ref{f1}, so we get that $e_1e_2\ne 0$ implies $s=1$ (this is true even for globally generated bundles with rank $r>2$). In other words, if $C$ is not connected, then we have either $e_1=0$ or $e_2=0$.

Now assume that the rank of $\Ee$ is $2$ and then we have $\omega_C \cong \Oo_C(0,0,-1)$. Thus each component $C_i$ of $C$ has genus at most $1$ and $C_i$ is an elliptic curve if and only if $e[i]_3=0$.
If the rank $2$ bundle $\Ee$ is decomposable with no trivial factor, then it is one of the following, up to reordering of the first and second factors.
\begin{enumerate}
\item $\Oo_X(2,2,0)\oplus \Oo_X(0,0,1)$; $c_2(\Ee)=(2,2,0)$
\item $\Oo_X(1,2,0)\oplus \Oo_X(1,0,1)$; $c_2(\Ee)=(2,1,2)$
\item $\Oo_X(2,1,1)\oplus \Oo_X(0,1,0)$; $c_2(\Ee)=(1,0,2)$
\item $\Oo_X(2,0,1)\oplus \Oo_X(0,2,0)$; $c_2(\Ee)=(2,0,4)$
\item $\Oo_X(1,1,1)\oplus \Oo_X(1,1,0)$; $c_2(\Ee)=(1,1,2)$
\end{enumerate}
Note that we have
\begin{align*}
\dim \Ext^1 (\Oo_X(1,2,0), \Oo_X(1,0,1))&=h^1(\Oo_X(-2,0,1))=2\\
\dim \Ext^1 (\Oo_X(0,2,0), \Oo_X(2,0,1))&=h^1(\Oo_X(2,-2,1))=6
\end{align*}
and in the other cases such extensions are always trivial.

\begin{lemma}\label{u1}
Let $\Ee$ be a globally generated vector bundle of rank $r\ge 2$ on $X$ with $c_1(\Ee)=(2,2,1)$, $c_2(\Ee)=(2,1,2)$ and no trivial factor. Then it fits into the sequence
$$0\to \Oo_X(1,0,1) \to \Ee \to \Oo_X(1,2,0)\to 0.$$
\end{lemma}

\begin{proof}
Since ${\pi _{12}}_{\vert_C}$ is an embedding and $\pi _{12}({C})\in |\Oo _{\PP ^1\times \PP^1}(2,1)|$, so $C$ is connected and rational. Since $\deg (\omega _C(0,0,1)) = \deg (\omega _C)+2 =0$, we get $\omega _C(0,0,1) \cong \Oo _C$. Since $s=1$, so we get $r=2$.

Since $h^0(\Oo _C(1,2,0)) = 5< 6=h^0(\Oo _X(1,2,0)$ and $h^1(\Oo _X(-1,0,-1)) =0$, we get $h^0(\Ii _C(1,2,0)) =1$. Since $h^1(\Oo _X(-1,0,-1)) =0$, we also get $h^0(\Ee (0,0,-1)) =1$. Thus there is a non-zero map $f: \Oo _X(1,0,1) \to \Ee$ and it induces the exact sequence,
$$0 \to \Oo _X(1,0,1)\to \Ee \to \Ii _T(1,2,0)\to 0$$
with $T$ a locally complete intersection subscheme of $X$ with pure codimension $2$, because $h^0(\Ii _C(0,2,0)) =h^0(\Ii _C(1,2,-1)) =0$. Since $(t_1+t_3)(t_1+2t_2) = 2t_1t_2+ t_1t_3+2t_2t_3$ and $c_2(\Ee ) =(2,1,2)$, so we get $T=\emptyset$.
\end{proof}

\begin{remark}
In \cite{CFM}, the moduli space of indecomposable, initialized ACM semistable bundles of rank $2$ with $c_1=(2,2,1)$ and $c_2=(2,1,2)$ is isomorphic to $\PP^1$. Indeed it is isomorphic to $\PP \Ext^1 (\Oo_X(1,2,0), \Oo_X(1,0,1))$. Each bundle in the moduli space is a pull-back of the bundle on a smooth quadric surface $Q$ fitting into the sequence
$$0\to \Oo_Q(0,1) \to \Ee \to \Oo_Q(2,0) \to 0$$
via $\pi_{23}^*$ twisted by $\Oo_X(1,0,0)$.
\end{remark}

\begin{proposition}\label{u1.1}
Let $\Ee$ be a globally generated vector bundle of rank $r\ge 2$ on $X$ with $c_1(\Ee)=(2,2,1)$, $c_2(\Ee)=(2,0,4)$ and no trivial factor. Then we have $r\in \{2,3\}$. To be precise, we have
\begin{itemize}
\item[(i)] $0\to \Oo_X(2,0,1) \to \Ee \to \Oo_X(0,2,0)\to 0,$ if $r=2$ and
\item[(ii)] $\Ee \cong \Oo _X(2,0,1)\oplus \Oo _X(0,1,0)^{\oplus 2}$, if $r=3$.
\end{itemize}
\end{proposition}

\begin{proof}
Since ${\pi _{12}}_{\vert_C}$ is an embedding and $\pi _{12}( C)\in |\Oo _{\PP ^1\times \PP^1}(0,2)|$, so $C$ has two connected components $C_1$ and $C_2$, both smooth and rational. Since $\omega _C(0,0,1)$ is globally generated, $\deg (\omega _{C_i}(0,0,1)) = \deg (\omega _{C_i})+e[i]_3 \ge 0$ for $i=1,2$, and $e[1]_3+e[2]_3=4$, so we get $e[1]_3=e[2]_3= 2$ and $\omega _C(0,0,1) \cong \Oo _C$. Since $\Ee$ has no trivial
factor, we get $2 \le r \le 3$ (see Remark \ref{24nov}).

\quad (a) Assume $r=2$. Since $h^0(\Oo _C(0,2,0)) = 2< 3=h^0(\Oo _X(0,2,0)$ and $h^1(\Oo _X(-2,0,-1))=0$, we get $h^0(\Ee(-2,0,-1))=h^0(\Ii _C(0,2,0)) \geq 1$. Thus there is a non-zero map $f: \Oo _X(2,0,1) \to \Ee$ and it induces the following exact sequence, because $h^0(\Ii _C(0,2,0)) =h^0(\Ii _C(1,2,-1)) =0$,
$$0 \to \Oo _X(2,0,1)\to \Ee \to \Ii _T(0,2,0)\to 0$$
with $T$ a locally complete intersection subscheme of $X$ with pure codimension $2$.
Since $(2t_1+t_3)(2t_2) = 4t_1t_2+2t_2t_3$ and $c_2(\Ee ) =(2,0,4)$, we get $T=\emptyset$.

\quad (b) Assume $r=3$. Since $h^0(\Oo _C(0,2,0))  = 2< 3=h^0(\Oo _X(0,2,0))$ and $h^1(\Oo _X(-2,0,-1))=0$, we get $h^0(\Ee(-2,0,-1))\geq 1$. Thus there is a non-zero map $f: \Oo _X(2,0,1) \to \Ee$. Set $\Ff:= \mathrm{coker}(f)$. Since $h^0(\Ii _C(0,2,0)) =h^0(\Ii _C(1,2,-1)) =0$, so $\Ff$ has no torsion. If $t\gg 0$, then $h^1(\Oo _X(2-t,-t,1-t))
= h^2(\Oo _X(2-t,-t,1-t))=0$. The exact sequence defining $\Ff$ gives $h^1(\Ff (-t)) =0$ for all $t\gg 0$. Hence $\Ff$ is reflexive by \cite[Remark 2.5.1]{Hartshorne1}. We have $c_3(\Ee ) = \deg (\omega _C(0,0,1)) = c_3(\Oo _X(2,0,1)\oplus \Oo _X(0,1,0)^{\oplus 2})$. The exact sequence defining $\Ff$ gives $c_3(\Ff )=0$. Hence $\Ff$ is locally free by \cite[Proposition 2.6]{Hartshorne1}. Since $c_1(\Ff )=(0,2,0)$, we get $\Ff \cong  \Oo _X(0,1,0)^{\oplus 2}$. Since $h^1(\Oo _X(2,-1,1)) =0$, we get
$\Ee \cong  \Oo _X(2,0,1)\oplus \Oo _X(0,1,0)^{\oplus 2}$.
\end{proof}

\begin{remark}\label{rema}
For each globally generated vector bundle $\Ff$ of rank $2$ with $c_1 = (0,2,1)$, the bundle $\Ee=\Ff(1,0,0)$ is globally generated with $c_1=(2,2,1)$. Note that $\Ff\cong \pi_{23}^* (\Gg)$ where $\Gg$ is a globally generated vector bundle on $Q\cong \PP^1 \times \PP^1$ with $c_1(\Gg)=(2,1)$. Such bundles are given in Example \ref{s2} or \cite[Proposition 3.7]{BHM2}. In particular we have
\begin{align*}
c_2(\Ee)=c_2(\Ff(1,0,0))&=c_2(\Ff)+(t_1t_3+2t_1t_2)\\
&=\pi_{23}^*(c_2(\Gg))+(t_1t_3+2t_1t_2),
\end{align*}
since $\Oo_X(1,0,0)\cdot \Oo_X(0,2,1)=t_1t_3+2t_1t_2$.
\end{remark}

\begin{remark}
Using the method in Remark \ref{rema} and the results in \cite{BHM2}, we may construct several decomposable bundles of rank $r\ge 3$ on $X$.
\begin{enumerate}
\item For the bundles $\Gg$ on $\PP^1 \times \PP^1$ as in Remark \ref{rema} we have $\Ee \cong \Oo_X(2,0,0)\oplus \pi_{23}^*(\Gg)$.
\item If $\Hh$ is a spanned bundle on $\PP^1\times \PP^1$ with $c_1(\Hh)=(2,2)$ and no trivial factor, then we also have $\Ee\cong \Oo_X(0,0,1)\oplus \pi_{12}^*(\Hh)$. For the list of such bundles, see \cite[Section $4$ and $5$]{BHM2}. The possible second Chern classes $c_2$ are in $\{3,4,5,6,8\}$.
\end{enumerate}
\end{remark}

Our main result in this section is the classification of globally generated vector bundles of rank $2$ on $X$ with $c_1=(2,2,1)$.

\begin{theorem}\label{kkkk2}
Let $\Ee$ be a globally generated vector bundle of rank $2$ on $X$ with $c_1=(2,2,1)$, $c_2=(e_1, e_2, e_3)$ and no trivial factor. If the associated curve $C$ is not connected with $s\ge 2$ components, then up to permutations on $(e_1, e_2)$ we have
\begin{enumerate}
\item $(s;e_1, e_2, e_3)=(2;2,0,4)$ ; $\Ee$ fits into the sequence
$$0\to \Oo_X(2,0,1) \to \Ee \to \Oo_X(0,2,0) \to 0,$$
\item $(s;e_1, e_2, e_3)=(3;3,0,6)$.
\end{enumerate}
If $C$ is connected, then up to permutations on $(e_1, e_2)$ we have
\begin{enumerate}
\item $C$ is an elliptic curve with $(e_1, e_2, e_3)=(2,2,0)$, or
\item $C$ is a rational curve with $e_3=2$ and
\begin{itemize}
\item [(i)] $(e_1, e_2, e_3)=(0,1,2)$ ; $\Ee \cong \Oo_X(1,2,1)\oplus \Oo_X(1,0,0)$,
 \item [(ii)] $(e_1, e_2, e_3)=(1,1,2)$ ; $\Ee \cong \Oo_X(1,1,1)\oplus \Oo_X(1,1,0)$,
\item [(iii)] $(e_1, e_2, e_3)=(2,1,2)$ ; $\Ee$ fits into the sequence
\begin{equation}\label{eqdddd1}
0\to \Oo_X(1,0,1) \to \Ee \to \Oo_X(1,2,0) \to 0,
\end{equation}
\item [(iv)] $(e_1, e_2, e_3)=(3,1,2)$,
\item [(v)] $(e_1, e_2, e_3)=(4,1,2)$.
\end{itemize}
\end{enumerate}
All these cases are realized by some globally generated bundles $\Ee$.
\end{theorem}

\begin{proof}
Until step (f) we assume that $C$ is connected, i.e. $s=1$. Recall that we have $p_a(C) \in \{0,1\}$.

If $C$ is an elliptic curve, i.e. $e_3=0$, then there are $p\in \PP^1$ and $C'\in |\Oo _{\PP^1\times \PP^1}(2,2)|$ such that $C = C'\times \{p\}$, because $C$ is connected. An equation of the divisor of $C'$ in $\PP^1\times \PP^1$ gives $\Ii _{C',\PP^1\times \PP^1}(2,2)\cong \Oo _C$ and thus $\Ii _{C',\PP^1\times \PP^1}(2,2)$ is globally generated. Since $\Oo _X(1,2,1)$ is globally generated and $h^1(\Oo _X(1,0,0)) =0$, so $\Ii _C(2,2,1)$ is also globally generated. In this case we have $h^0(\Ee)=h^0(\Ii_C(2,2,1))+1=11$.

If $C$ is rational, then we have $e_3=2$ since $\omega _C \cong \Oo _C(0,0,-1)$. Since ${\pi _{12}}_{\vert_C}$ is an embedding and $C$ is smooth, connected and rational, so we have either $e_1=1$ or $e_2=1$. Without losing generality we may assume $e_2=1$. Since $Y$ has multidegree $(4,4,8)$, then we have $0\le e_1 \le 4$.

\quad (a) Assume $e_1=0$ and then we have $(e_1, e_2, e_3)=(0,1,2)$. There are $p\in \PP^1$ and $C'\in |\Oo_{\PP^1 \times \PP^1}(2,1)|$ such that $C=\{p\} \times C'$. Here we have $\deg (\Oo_C(1,0,0))+1=2=h^0(\Oo_X(1,0,0))$, $h^1(\Oo_X(-1,-2,-1))=0$ and so $h^0(\Ee(-1,-2,-1))=h^0(\Ii_C(1,0,0))=1$. Hence $\Ee$ fits into an exact sequence
\begin{equation}\label{eqa10}
0\to \Oo_X(1,2,1) \to \Ee \to \Ii_T (1,0,0) \to 0
\end{equation}
with either $T=\emptyset$ or $T$ a locally complete intersection curve. We have $T=\emptyset$, because $c_2(\Oo_X(1,2,1)\oplus \Oo_X(1,0,0))=t_1t_3+2t_1t_2$. Since $h^1(\Oo_X(0,2,1))=0$, so the extension (\ref{eqa10}) is trivial.

\quad (b) Assume $e_1=1$. Since $\deg (\Oo _C(1,1,0)) +1=3 =h^0(\Oo _X(1,1,0)) -1$, so we have $h^0(\Ii _C(1,1,0))\ge 1$. Since $h^1(\Oo _X(-1,0,-1)) =0$, we also get $H^0(\Ee (-1,-1,-1)) \ne 0$. Now since $h^0(\Ii _C(1,0,0)) = h^0(\Ii _C(0,1,0)) =h^0(\Ii _C(1,1,-1))=0$ , we get an exact sequence
\begin{equation}\label{eqm1}
0 \to \Oo _X(1,1,1)\to \Ee \to \Ii _T(1,1,0)\to 0
\end{equation}
with either $T=\emptyset$ or $T$ a locally complete intersection curve. We have  $T= \emptyset$, because $c_2(\Oo _X(1,1,1)\oplus \Oo _X(1,1,0)) = t_3t_1+t_3t_2 +2t_1t_2$.
Since $h^1(\Oo _X(0,0,1)) =0$, we get $\Ee \cong  \Oo _X(1,1,1)\oplus \Oo _X(1,1,0)$. In particular we have $c_2(\Ee)=(1,1,2)$.

\quad (c) Assume $e_1=2$ and then we have $(e_1, e_2, e_3)=(2,1,2)$. This is the case in Lemma \ref{u1}.

\quad (d) Assume $e_1=3$. We have $h^0(\Oo _C(1,1,1)) = 7$ and so $h^0(\Ii _C(1,1,1)) >0$. Since $h^1(\Oo _X(-1,-1,0))=0$, we get $h^0(\Ee (-1,-1,0)) =0$. Since $h^0(\Ii _C(0,1,1)) = h^0(\Ii _C(1,0,1)) =h^0(\Ii _C(1,1,0)) =0$, we get that $\Ee$ fits in the following exact sequence
\begin{equation}\label{eqm2}
0 \to \Oo _X(1,1,0)\to \Ee \to \Ii _T(1,1,1)\to 0
\end{equation}
with either $T=\emptyset$ or $T$ a locally complete $1$-dimensional subscheme with $\omega _T \cong \Oo _T(-2,-2,1)$ and multidegree $(2,0,0)$, because of $c_2(\Oo_X(1,1,1)\oplus \Oo_X(1,1,0))=(1,1,2)$. Conversely, take two distinct points $p_1,p_2\in \PP^1\times \PP^1$ and set $T:= \PP^1\times \{p_1,p_2\}$. Then we get a vector bundle $\Ee$ fitting into the sequence (\ref{eqm2}). Now assume that $p_1, p_2$ are not contained in the same ruling either of $|\Oo _{\PP^1\times \PP^1}(0,1)|$ or of $|\Oo _{\PP^1\times \PP^1}(1,0)|$. Since $\Ii _{\{p_1,p_2\},\PP^1\times \PP^1}(1,1)$ is globally generated, there are divisors $E_1,E_2\in |\Oo _{\PP^1\times \PP^1}(1,1)|$ with $E_1\cap E_2 =\{p_1,p_2\}$ as schemes. Set $H_i:= \PP^1\times E_i$ for $i=1,2$. Since $T = H_1\cap H_2$ as schemes and $H_i\in |\Oo _X(0,1,1)|$, the sheaf $\Ii _T(0,1,1)$
is globally generated and so is $\Ii _T(1,1,1)$. It implies that any bundle $\Ee$ in (\ref{eqm2}) is globally generated.

\quad (e) Assume $e_1=4$. If such a curve $C$ exists, then $h^0(\Oo _C(1,2,1)) = 9<12 = h^0(\Oo _X(1,2,1))$. We also see that $h^0(\Ii _C(0,2,1)) =h^0(\Ii _C(1,2,0)) =0$ and that $h^0(\Ii _C(1,1,1)) < 3 \le h^0(\Ii _C(1,2,1))$. Indeed, for a general $C$, we even have $h^0(\Ii _C(1,1,1)) = h^1(\Ii _C(1,1,1)) =0$. Since $h^1(\Oo _X(-1,0,0)) =0$, we get that $\Ee$ fits in the following exact sequence
\begin{equation}\label{eqm3}
0 \to \Oo _X(1,0,0)\to \Ee \to \Ii _T(1,2,1)\to 0
\end{equation}
with either $T=\emptyset$ or $T$ a locally complete intersection $1$-dimensional scheme with $\omega _T \cong \Oo _T(-2,0,-1)$. Since $t_1(t_1+2t_2+t_3) = t_1t_3+2t_1t_2$ and $c_2(\Ee  )$ is represented by a curve with multidegree $(4,1,2)$, we have
$$\deg (\Oo _T(1,0,0)) =4~,~\deg (\Oo _T(0,1,0)) =  \deg (\Oo _T(0,0,1)) =0.$$
To prove the existence of $\Ee$ we reverse the construction, because any $T$ as above with $\Ii _T(1,2,1)$ globally generated gives a globally generated bundle $\Ee$. Take any complete intersection $Z\subset \PP^1\times \PP^1$ of two elements $D_1,D_2\in |\Oo _{\PP^1\times \PP^1}(2,1)|$ and set $T = \PP^1\times Z$. $\Ii_Z(2,1)$ is globally generated and so is $\Ii_T(0,2,1)$. In particular $\Ii_T(1,2,1)$ is globally generated.

\quad (f) Assume $s>1$ and so we have $e_1e_2=0$. Without losing generality we may assume $e_2=0$. We do this case even when $s=1$. We have $e_1 =s$ and $e_4\le 8$. Since $\omega _C \cong \omega _C(0,0,1)$ and $C_i$ is smooth and rational, we have $e[i]_3=2$ and $e[i]_1=1$ for all $i$.
Take the intersection $Y\subset X$ of two general elements of $|\Ii _C(2,2,1)|$ and then each connected component of $C$ appears with multiplicity one in $Y$. By the Bertini theorem we have $Y = C\cup D$ with $D$ a reduced curve smooth outside $C\cap D$ and
$$\deg (\Oo _D(1,0,0)) = 4-s~,~\deg (\Oo _D(0,1,1)) =4~,~\deg (\Oo _D(0,0,1)) = 8-e_3.$$

\quad {\emph {Claim }}: $s\le 3$.

\quad {\emph {Proof of Claim }}: If $s =4$, then we have $e_3=8$ and so $D$ is the union of $4$ distinct fibers of $\pi _{23}$, i.e. there are four distinct points $p_1,p_2,p_3,p_4\in \PP^1\times \PP^1$ with $D = \PP^1 \times \{p_1,p_2,p_3,p_4\}$. If $D_i$'s are connected components of $D$, then we have $\deg (C \cap D_i)\le 2$, because $\Ii _C(2,2,1)$ is globally generated. Since $C \cup D$ has $8$ irreducible components and each of them is smooth and rational, we have $p_a(Y) \le 1$, a contradiction. Thus we have $s\in \{2,3\}$. \qed

\quad (f1) If $s=2$, then we have $(e_1, e_2, e_3)=(2,0,4)$ and we are in the case of Lemma \ref{u1.1}.

\quad (f2) If  $s=3$, then we have $(e_1, e_2, e_3)=(3,0,6)$ and we may use the case $s=1$ and $(e_1,e_2,e_3) = (1,4,2)$ proved in step (e); it gives bundle for all ranks at most $h^0(\omega _C(0,0,1))+1$. Let $\Gg$ be a spanned bundle of rank $3$ with $c_1(\Gg ) =(2,2,1)$ and with associated curve with multidegree $(1,4,2)$. Then we have $h^0(\Gg ) \ge 6$. Let $V\subseteq H^0(\Gg )$ be a general $6$-dimensional linear subspace. Since $\dim (X) = 3$, the evaluation map $\tau : V\otimes \Oo _X \to \Gg$ is surjective. Set $\Hh := \ker (\tau )^\vee$ and then $\Hh$ is a spanned bundle of rank $3$. The value of $c_2(\Gg)$ gives that $\Hh$ is associated to a curve $C$ with multidegree $(3,0,6)$. Since $\Ii _C(2,2,1)$ is spanned and $C$ has multidegree $(3,0,6)$, it has $s=3$ and so this case is realized.
\end{proof}

By Remark \ref{24nov} the case $s\ge 2$ in Theorem \ref{kkkk2} gives the list of all spanned $\Ee$ with $r\ge 2$, $c_3(\Ee )=0$ and no trivial factor; they have $r\le 3$ if $s=2$, and $r\le 4$ if $s=3$. The proof of Theorem \ref{kkkk2} gives that some families of bundles are parametrized by an irreducible family of curves $C$; the case with $s=2$, the case with $C$ an elliptic curve and case (i) and (ii) for a rational curve. In the other cases the proof gives that the family is irreducible; see step (e) for $(e_1,e_2,e_3) =(4,1,2)$ and step (f2) for the case $s=3$.

\begin{remark}\label{ddd1}
Take $\Ee$ as in Theorem \ref{kkkk2} with $s=1$ and  $(e_1, e_2, e_3)=(2,1,2)$. From (\ref{eqdddd1}) we get $h^1(\Ee (t,t,t)) = h^2(\Ee (t,t,t)) = 0$ for all $t\in \ZZ$. Hence $\Ee$ is ACM. The indecomposable ones, i.e. the ones for which (\ref{eqdddd1}) does not split, are the ones in case (5) of \cite[Theorem B]{CFM0}.
\end{remark}

\begin{remark}\label{ddd2}
Take $\Ee$ as in Theorem \ref{kkkk2} with $s=1$ and $(e_1, e_2, e_3)=(3,1,2)$. From (\ref{eqm2}) we get $h^1(\Ee (-1,-1,-1)) >0$. Look at step (d) of the proof of  Theorem \ref{kkkk2} and assume that $T$ is induced by two distinct points $p_1,p_2$. Since we have $\langle T \rangle \cong \PP^3$, we get $h^1(\Ee (t,t,t)) =0$ for all $t\ne -1$ and $h^1(\Ee (-1,-1,-1)) =1$.
\end{remark}

From now on we take a smooth dependency locus $C\subset X$ of a globally generated bundle $\Ee$ of rank $r\ge 3$. The Hartshorne-Serre condition says that $\omega _C$ is globally generated and so we have $p_a(C_i) >0$.

\begin{theorem}\label{r1}
For globally generated vector bundles of higher rank on $X$ with $c_1=(2,2,1)$ we have the following:
\begin{enumerate}
\item There exists a globally generated vector bundle $\Ee$ of rank $r$ with no trivial factor if and only if $2 \le r \le 17$.
\item For each $r$ with $3\le r\le 17$ there is a bundle as in  Example \ref{ooo0} or Remark \ref{rem3.3.3}. In particular a general dependency locus $C$ is connected and with multidegree $(4,4,8)$, i.e. $c_2(\Ee ) =8t_1t_2+4t_1t_3+4t_2t_3$.
\item If $r\in \{14,15,16,17\}$, then each $\Ee$ comes from Remark \ref{rem3.3.3}.
\end{enumerate}
\end{theorem}

\begin{proof}
The case $r=2$ is covered by $\Oo _X(2,2,0)\oplus \Oo _X(0,0,1)$. If $3 \le r \le 17$, the existence is given either by  Example \ref{ooo0} or Example \ref{ooo1}.

Fix $r \in \{14,15,16,17\}$ and assume the existence of $\Ee$ not coming from Remark \ref{rem3.3.3} with a dependency locus $C\subset X$ of $r-1$ general sections of $\Ee$. Let $Y =C\cup D$ be the intersection of $2$ general elements of $|\Ii _C(2,2,1)|$. By the Hartshorne-Serre correspondence it is sufficient to prove that $h^0(\omega _C(0,0,1)) \le 12$ if $\Ee$ does not come from Remark \ref{rem3.3.3}, i.e. if $D\ne \emptyset$. Note that $D$ is reduced with multidegree $(4-e_1,4-e_2,8-e_3)$.

\quad ({a}) Assume $s=1$ and set $g:= p_a({C})$. Since $\pi _{12}({C}) \cong C$, we have $g = e_1e_2-e_1-e_2+1$. Since $\omega _C(0,0,1)$ is spanned, we have $e_3\ne 1$. If $e_3=0$, then we get $h^0(\omega _C(0,0,1)) = h^0(\omega _C) = g\le 9$. If $e_3\ge 2$, then Riemann-Roch gives $h^0(\omega _C(0,0,1)) = g+e_3-1 =e_1e_2-e_1-e_2+e_3$.

\quad ({a1}) Assume $(e_1,e_2) =(4,4)$. In this case $D$ has multidegree $(0,0,8-e_3)$. So $D$ has $8-e_3$ connected components $D_1,\dots ,D_{8-e_3}$ and there are $o_j\in \PP^1\times \PP^1$ for each $1\le j \le 8-e_3$ such that $D _j = \{o_j\}\times \PP^1$. We have $o_i\ne o_j$ for all $i\ne j$. Since $C\cup D$ is a complete intersection of two ample divisors, it is connected. Since $p_a({C})=p_a(C\cup D)$, each $D_j$ meets quasi-transversally $C$ at a unique point, say $Q_j$.

\quad {\emph {Claim }}: A general $S\in |\Ii _C(2,2,1)|$ is singular.

\quad {\emph {Proof of Claim }}: Assume that $S$ is smooth. We have $C\cup D\in |\Oo _S(2,2,1)|$ and $\{Q_1\}$ is the scheme-theoretic intersection of $D_1$ and $Y-D_1$, i.e. $D_1\cdot (Y-D_1) =1$ the intersection number in the smooth surface $S$. Hence the very ample line bundle $\Oo _S(2,2,1)$ is not $2$-connected in the sense of \cite{bdl,bl1,ven,bl2}. We have $\Oo _S(2,2,1)\cdot \Oo _S(2,2,1) = (2t_1+2t_2+t_3)^3 = 12>4$. Since $S$ contains only finitely many curves $\{o\}\times \PP^1$ with $o\in \PP^1\times \PP^1$, so the pair $(S,\Oo _S(2,2,1))$ is not a scroll in the sense of \cite{bl1,ven}. Since $(S,\Oo _S(2,2,1))$ is not $2$-connected, \cite[Theorem A]{bl1} gives a contradiction. \qed

Since $C$ is a smooth curve, $\dim (X) =2+\dim ({C})$ and $C$ is the scheme-theoretic base locus of $\Ii _C(2,2,1)$, Claim 1 contradicts a strong form of Bertini's theorem for linear systems with a smooth base locus in \cite[Theorem 2.1]{diaz}. In summary, this contradiction rules out all the cases with $s=1$ and multidegrees $(4,4,e_3)$ with $0 \le e_3\le 7$.

\quad ({a2}) If $(e_1,e_2)\ne (4,4)$, then we have $g\le 6$ and so $h^0(\omega _C(0,0,1)) \le 6+e_3-1\le 13$, with equality only if $(e_1,e_2) \in \{(3,4),(4,3)\}$ and $e_3=8$. In this case $D$ has either multidegree $(1,0,0)$ or $(0,1,0)$. In particular $D$ is smooth and rational. Since $p_a(C\cup D)=9$, we have $\deg (D\cap C)>2$ and so $\Ii _C(2,2,1)$ is not globally generated, a contradiction. Thus this case may give solutions only with $r\le 13$.

\quad ({b}) Assume $s>1$. We saw that each $C_i$ is smooth and rational with $e[i]_3\ge 2$ for all $i$. We have $h^0(\omega _C(0,0,1)) = \sum _{i=1}^{s} (e[i]_3-1) \le 8-s \le 6$. If $C$ is not connected, we get $r\le 7$.
\end{proof}

We give several examples of globally generated vector bundles of rank at least $3$ on $X$ with $c_1=(2,2,1)$.

\begin{example}\label{s5}
Since $h^1(\Oo _X(-2,1,0))=2$ and $h^1(\Oo _X(-2,1,1))=4$, there are non-trivial extensions
\begin{equation}\label{equ2}
0 \to \Oo _X(0,1,1)\oplus \Oo _X(0,1,0) \to \Ee \to \Oo _X(2,0,0)\to 0
\end{equation}
and $\Ext^1$ is a $6$-dimensional vector space. So we get indecomposable bundles with $c_1=(2,2,1)$ and $c_2=(2,1,4)$. We may construct these bundles also as extensions of $\Ee$ in Example \ref{s1} by $\Oo_X(0,1,0)$.
\end{example}

\begin{example}\label{s7}
Since $h^1(\Oo _X(-2,0,1))=2$ and $h^1(\Oo _X(-2,2,0))=3$, there are non-trivial extensions
\begin{equation}\label{equ2}
0 \to \Oo _X(0,0,1)\oplus \Oo _X(0,2,0) \to \Ee \to \Oo _X(2,0,0)\to 0
\end{equation}
and $\Ext^1$ is a $5$-dimensional vector space. So we get indecomposable bundles with $c_1=(2,2,1)$ and $c_2=(2,2,4)$.
\end{example}

\begin{example}
For the bundles $\Ff_{\lambda}$ in Example \ref{s4}, we may compute $h^1(\Ff_{\lambda}^\vee(0,1,0))\geq 4$ so we have  a family $\{\Hh _\nu \}$ of indecomposable bundles with $c_1=(2,2,1)$ and $c_2=(3,3,5)$.
\end{example}

\begin{example}\label{s6}
Let $\Gg=\phi^*(T\PP^3(-1))$, where $\phi : X \to \PP^3$ is a linear projection from a linear subspace $W\subset \PP^7$ with $\dim (W) = 3$ and $W\cap X=\emptyset$.
From the Euler sequence of $T\PP^3$ we get an exact sequence
\begin{equation}\label{eqeeaa1}
0 \to \Gg ^\vee \to \Oo _X^{\oplus 4} \to \Oo _X(1,1,1) \to 0
\end{equation}
From (\ref{eqeeaa1}) we get $\dim \Ext^1 (\Gg, \Oo(1,1,0))= h^1(\Gg^\vee(1,1,0)) \geq 18 - 16 =2$ and so we have a family $\{\Ff _\lambda \}$ of non-trivial extensions of  $\Gg$ by $\Oo(1,1,0)$ with
$c_1=(2,2,1)$ and $c_2=(3,3,4)$ with the exact sequence
\begin{equation}\label{eeqqaa}
0 \to \Oo_X^{\oplus 3} \to \Ff_{\lambda} \to \Ii_C(2,2,1) \to 0,
\end{equation}
where $C$ is a smooth curve of multidegree $(3,3,4)$, and the exact sequence
\begin{equation}\label{eqeeaa2}
0 \to \Oo _X(1,1,0) \to \Ff _{\lambda} \to \Gg \to 0.
\end{equation}
Since $h^2(\Oo _X(-2,-2,-1)) =0$, from the dual of (\ref{eqeeaa1}) we get $h^0(\Gg (-1,-1,0)) = h^1(\Gg (-1,-1,0)) =0$.
From (\ref{eqeeaa2}) we also get $h^0(\Ff_{\lambda}(-1,-1,0)) =1$ and $h^1(\Ff_{\lambda}(-1,-1,0)) =0$. From the sequence (\ref{eeqqaa}) we get $h^0(\Ii_C(1,1,1))=1$, $h^1(\Ii _C(1,1,1)) =1$ and so $h^0(\Oo_C(1,1,1))=7$. It implies $p_a({C})=4$. Let $\Ll:=\omega_C(-1,-1,0))$ be a line bundle on $X$ of degree $0$. From the exact sequence
$$0\to \Ii_C(1,1,0) \to \Oo_X(1,1,0) \to \Oo_C(1,1,0) \to 0$$
together with $h^0(\Ii_C(1,1,0))=0$ from (\ref{eeqqaa}), we have $h^0(\Oo_C(1,1,0))\ge 4$. By Riemann-Roch, we have $h^0(\Ll)-h^1(\Ll)=3$ and so $h^0(\Ll)\ge 1$ since $h^1(\Ll)=h^0(\Oo_C(1,1,0))$. Thus we have $h^0(\Ll)=1$ and $\omega_C (-1,-1,0)\cong \Oo_C$. Now the Hartshorne-Serre correspondence implies the existence of a globally generated vector bundle $\Hh$ fitting into the sequence
$$0\to \Oo_X \to \Hh \to \Ii_C(3,3,2) \to 0.$$
Note that $h^0(\Hh(-2,-2,-1))=h^0(\Ii_C(1,1,1))=h^0(\Ff_{\lambda}(-1,-1,0))=1$. From the sequence (\ref{eeqqaa}), we also have $h^0(\Hh(-3,-2,-1))=h^0(\Hh(-2,-3,-1))=h^0(\Hh(-2,-2,-2))=0$ and so a non-zero section in $H^0(\Hh(-2,-2,-1))$ induces an exact sequence
$$0\to \Oo_X(2,2,1) \to \Hh \to \Ii_T(1,1,1) \to 0$$
with $T$ a locally complete intersection $0$-dimensional subscheme of codimension $2$. Since $c_2(\Hh)=(3,3,4)$, so we get $T=\emptyset$. Since $h^1(\Oo_X(1,1,0))=0$, so the extension is trivial. Thus we have $\Hh \cong \Oo_X(2,2,1)\oplus \Oo_X(1,1,1)$ and $\Ii_C$ admits the following locally free resolution:
$$0\to\Oo_X(-1,-1,-1)\to\Oo_X\oplus\Oo_X(1,1,0)\to\Ii_C(2,2,1)\to 0.$$
Moreover $h^1(\Ff_{\lambda}^\vee)=4$ and so we have higher rank bundles up to $r=8$ with the same Chern classes and no trivial factor.
\end{example}


\section{Case of $c_1=(2,2,2)$}
Let $\Ee$ be a globally generated vector bundle of rank $r\ge 2$ on $X$ with $c_1(\Ee)=(2,2,2)$ and the associated curve $C = C_1\sqcup \cdots \sqcup C_s$ where each $C_i$ is a smooth and connected component. If $Y$ is the complete intersection of two elements in $|\Oo_X(2,2,2)|$, then it has
Chern polynomial $(1+2t_1+2t_2+2t_3)^2$ and hence it has has multidegree $(8,8,8)$. We also have $\omega _Y \cong \Oo _Y(2,2,2)$ and so $h^1(\Oo _Y) =p_a(Y) = 25$. The Hartshorne-Serre condition says that $\omega _C$ is spanned, i.e. $p_a(C_i)>0$ for all $i$. If $r=2$, then $\omega _C \cong \Oo _C$ and hence each $C_i$ is an elliptic curve.

From now on we take a smooth dependency locus $C\subset X$ of a globally generated bundle $\Ee$ of rank $r\ge 3$. The Hartshorne-Serre condition says that $\omega _C$ is globally generated and so we have
$p_a(C_i) >0$.

\begin{theorem}\label{r2}
For globally generated vector bundles of higher rank on $X$ with $c_1=(2,2,2)$ we have the following:
\begin{enumerate}
\item There exists a globally generated vector bundle $\Ee$ of rank $r$ with no trivial factor if and only if $2 \le r \le 26$.
\item For each $3\le r\le 26$, there is a bundle as in Example \ref{ooo0} or Remark \ref{rem3.3.3}. In particular a general dependency locus $C$ is connected with multidegree $(8,8,8)$, i.e. $c_2(\Ee ) =8t_1t_2+8t_1t_3+8t_2t_3$.
\item If $r \in \{24,25,26\}$, then each $\Ee$ comes from Remark \ref{rem3.3.3}.
\end{enumerate}
\end{theorem}

\begin{proof}
Note that the case $r=2$ is covered by $\Oo _X(2,2,0)\oplus \Oo _X(0,0,2)$. If $3 \le r \le 26$, the existence is given either by  Example \ref{ooo0} or by Remark \ref{rem3.3.3}.

Fix $r \in \{24,25,26\}$ and assume the existence of $\Ee$ not coming from Remark \ref{rem3.3.3} with  a dependency locus $C\subset X$ of $r-1$ general sections of $\Ee$. Let $S$ be a general element of $|\Ii _C(2,2,2)|$. Since $C$ is a smooth curve, we get that $S$ is a smooth surface by Diaz-Harbater form of the Bertini theorem in \cite[Theorem 2.1]{diaz}. The adjunction formula gives $\omega _S \cong \Oo _S$. Let $T\subset S$ be a line of $S\subset \PP^7$, if any. The adjunction formula gives $\omega _T \cong \Oo _S(T)_{\vert_T}$. Since $\deg (\omega _T)=-2$, we get
$T^2=-2<0$ and so $S$ contains only finitely many lines.

Let $Y =C\cup D$ be the intersection of $S$ with a general element of $|\Ii _C(2,2,2)|$. By the Hartshorne-Serre correspondence it is sufficient to prove that $h^0(\omega _C )\le 22$ if $\Ee$ does not come from Remark \ref{rem3.3.3}, i.e. $D\ne \emptyset$. Now by duality it is sufficient to prove that $h^1(\Oo _C)\le 22$ if $D\ne \emptyset$. We have $h^1(\Oo _C) = \sum _{i=1}^{s} h^1(\Oo _{C_i})$. Note that $D$ is reduced with multidegree $(8-e_1,8-e_2,8-e_3)$. Recall that $h^1(\Oo _Y)=25$ and so it is sufficient to prove
that $h^1(\Oo _C)\le h^1(\Oo _Y)-3$.

\quad {\emph {Claim }}: The line bundle $\Oo _S(2,2,2)$ is $4$-connected.

\quad {\emph {Proof of Claim }}: Set $\Ll := \Oo _S(1,1,1)$. Since $S\in |\Oo _X(2,2,2)|$, so we have $\Oo _S(2,2,2)\cdot \Oo _S(2,2,2) = (2t_1+2t_2+2t_3)^3 = 24 >17$. Since $\Ll$ is very ample and $\Oo _S(2,2,2) = \Ll^{\otimes 2}$, for every irreducible curve $T\subset S$ we have
$$2\le T\cdot \Oo _S(2,2,2)\equiv 0 \pmod 2,$$
and $T\cdot \Oo _S(2,2,2) =2$ if and only if $T$ is a line contained in $S$. Since $T\cdot \Oo _S(2,2,2)$ is an even integer, the pair $(S,\Oo _S(2,2,2))$ is neither a $\PP^1$-bundle over a curve nor a cubic scroll over a curve. By \cite[Theorem A]{bl1}, the line bundle $\Oo _S(2,2,2)$ is $2$-connected. Since $T^2=-2$ for each line $T$ of $S$ and $\deg (T\cdot \Oo _S(2,2,2)) \ge 4$ for each curve $T\subset S$, which is not a line, so \cite[Theorem B]{bl1} and \cite{bl2} give first that $\Oo _S(2,2,2)$ is $3$-connected and then \cite[Theorem C]{bl1} gives that $\Oo _S(2,2,2)$ is $4$-connected. \qed

Write $D[0]:= D$ and $D[s-1]:= D\cup C_1\cup \cdots \cup C_{s-1}$ if $s\ge 2$. Let $E_1,\dots ,E_t$ be the connected components of $D[s-1]$. Since $\Oo _S(2,2,2)$ is $4$-connected  by the \emph{Claim} and $E_i\cap E_j=\emptyset$ for all $i\ge j$, we have $\deg (D[s-1]\cap C_s) \ge 4t$. Since $h^0(\Oo _{D[s-1]}) =t$ and $h^2(\Oo _Y)=0$, the Mayer-Vietoris exact sequence
$$0 \to \Oo _Y\to \Oo _{D[s-1]}\oplus \Oo _{C_s} \to \Oo _{D[s-1]\cap C_s}\to 0$$gives $h^1(\Oo _Y ) \ge 4t-t+h^1(\Oo _{C_s}) +h^1(\Oo _{D[s-1]}) \ge 3+ h^1(\Oo _{C_s}) +h^1(\Oo _{D[s-1]})$. If $s=1$, then we use that $h^1(\Oo _D)\ge 0$. Now assume $s\ge 2$. Let $A, B$ be projective curves with $A\subset B$. Since $\Ii _{A,B}$ is supported by a subscheme of $B$, we have $h^2(B,\Ii _{A,B})=0$. Hence the exact sequence
$$0 \to \Ii _{A,B}\to \Oo _B\to \Oo _A\to 0$$gives that the natural map $H^1(\Oo _B) \to H^1(\Oo _A)$ is surjective. With $A=C_1\cup \cdots \cup C_{s-1}$ and $B=D[s-1]$, we have $h^1(\Oo _{D[s-1]}) \ge h^1(\Oo _{C_1\cup \cdots \cup C_{s-1}})$ and so $h^1(\Oo _Y)\ge 3+h^1(\Oo _C)$.
\end{proof}

\begin{remark}\label{v1}
Take a smooth $C$ with $\Ii _C(2,2,2)$ and $\omega _C$ globally generate. Each $C_i$ has positive genus. Since $C_i$ is not rational, we have $e[i]_j\ne 1$ for all $j=1,2,3$. For fixed $i\in \{1,\dots ,s\}$, assume for the moment $e[i]_j=0$ for some $j\in \{1,2,3\}$. Just to fix the notation we assume $j=1$. Then there are $p\in \PP^1$ and $C'_i\in |\Oo _{\PP^1\times \PP^1}(e[i]_3,e[i]_2)|$ such
that $C_i = \{p\}\times C'_i$. Since $C'_i\subset \PP^1\times \PP^1$ has positive genus and $\{p\}\times \PP^1\times \PP^1$ is not in the base locus of $\Ii _C(2,2,2)$, we get $e[i]_2=e[i]_3=2$. Now assume $e[i]_j\ne 0$ for all $j$. Then we get $e[i]_j\ge 2$ for all $j$.

Assume $s\ge 2$ and $e[1]_3=0$.

\quad {\emph {Claim 1}}: $e_3=0$, i.e. $e[i]_3=0$ for all $i>1$.

\quad {\emph {Proof of Claim 1}}: Assume for instance $e[2]_3>0$. Since $e[1]_3=0$, there are $p\in \PP^1$ and $C'_1\in |\Oo _{\PP^1\times \PP^1}(2,2)|$ such that $C_1 = C'_1\times \{p\}$. Since $e[2]_3>0$, the scheme $Z:= C_2\cap \PP^1\times \PP^1\times \{p\}$ is non-empty (if $0$-dimensional, it has degree $e[2]_3$). Write $Z = Z'\times \{p\}$. Since $C_1\cap C_2 =\emptyset$ and $Z'\ne \emptyset$, we have $h^0(\PP^1\times \PP^1,\Ii _{Z'\cup C'_1,\PP^\times \PP^1}(2,2)) =0$. Hence $\PP^1\times \PP^1 \times \{p\}$ is in the base locus of $\Ii _C(2,2,2)$, a contradiction. \qed

By \emph{Claim 1}, if there are $i, j$ such that $1\le i \le s$ and $1\le j \le 3$ with $e[i]_j=0$, then we have $e_j=0$ and $e[i]_k=2$ for all $i$ and $k\ne j$. If $e[i]_j \ne 0$ for all $i, j$,
then we have $e[i]_j \ge 2$ for all $i, j$.

\quad {\emph {Claim 2}}: Assume $e[i]_j \ne 0$ for all $i, j$. Then $s\le 3$.

\quad {\emph {Proof of Claim 2}}: Since $e[i]_j\ge 2$ for all $i, j$ by \emph{Claim 1} and $Y$ has multidegree $(8,8,8)$, so we have $s\le 4$ and if $s=4$, then $C=Y$. This is not possible, because
$h^0(\Oo _Y)=1$ (see Remark \ref{rem3.3.3}). \qed

\quad {\emph {Claim 3}}: We have $s\le 3$.

\quad {\emph {Proof of Claim 3}}: By \emph{Claim 2} we may assume $e[i]_j=0$ for some $i, j$, say $e[1]_3=0$. By \emph{Claim 1} we have $e[i]_3=0$ and $e[i]_1=e[i]_2 =2$
for all $i$ and there are $o_i\in \PP^1$ and $C'_i\in  |\Oo _{\PP^1\times \PP^1}(2,2)|$ such
that $C_i = C'_i\times \{o_i\}$. Since $Y$ has multidegree $(8,8,8)$, we get $s\le 4$. Assume $s=4$ and write $Y = C\cup D$. We get that $D$ is the disjoint union of $8$ lines with
multidegree $(0,0,1)$, say $D = D_1\cup \cdots \cup D_8$. Since $\Ii _C(2,2,2)$ is globally generated, we get $\deg (D_i\cap C)\le 2$ for all $i$. Hence we have $p_a(Y) \le -3+8$, a contradiction. \qed
\end{remark}

\begin{example}\label{v4}
Let $Z\subset \PP^1\times \PP^1$ be any $0$-dimensional subscheme of degree $2$ and set $T:= Z\times \PP^1$. Since $\omega _T \cong \Oo _T(-2,-2,-2)$, the Hartshorne-Serre correspondence gives a unique bundle $\Ee$ fitting into an exact sequence
\begin{equation}\label{eqv2}
0\to \Oo _X(1,1,1)\to \Ee \to \Ii _T(1,1,1)\to 0.
\end{equation}
Any such a bundle $\Ee$ is globally generated if and only if $Z$ is the complete intersection of two elements of $|\Oo _{\PP^1\times \PP^1}(1,1)|$, i.e. $Z$ is not contained in one of the rulings of $\PP^1\times \PP^1$.
\end{example}

\begin{proposition}\label{v2}
We have $c_2(\Ee)=(4,4,0)$ if and only if we have either
\begin{itemize}
\item [(i)] $r=2$ and $0\to \Oo _X(2,2,0)\to \Ee \to \Oo _X(0,0,2)\to 0$ , or
\item [(ii)] $r=3$ and $\Ee \cong \Oo _X(2,2,0)\oplus \Oo _X(0,0,1)^{\oplus 2}$.
\end{itemize}
\end{proposition}

\begin{proof}
The ``~if~'' part is obvious, because $(1+t_3)^2 = 1+2t_3$ and $(1+2t_1+2t_2)(1+2t_3) = 1+2t_1+2t_2+2t_3+4t_1t_3+4t_2t_3$. Now assume $(e_1,e_2,e_3) =(4,4,0)$. Since $e[i]_3=0$ for all $i$, Remark \ref{v1} gives $s=2$ and $e[i]_j=2$ for all $i, j \in \{1,2\}$. Hence each connected component of $C$ is an elliptic curve and so $\omega _C\cong \Oo _C$ and $c_3(\Ee )=0$. Since $s=2$, Remark \ref{24nov} gives $r\in \{2,3\}$.

First assume $r=2$. Note that $2=h^0(\Oo _X(0,0,2)) -1=h^0(\Oo_C)=h^0(\Oo _C(0,0,2))$ and so we have $h^0(\Ii _C(0,0,2)) >0$. Since $s\ge 2$ and $C_i\ne \PP^1$, we have $h^0(\Ii _C(0,0,1)) =0$. Since $h^0(\Ii _C(-1,0,2)) = h^0(\Ii _C(0,-1,2)) =0$ and $c_2(\Ee )=4t_2t_3+4t_1t_3$,
we get that $\Ee$ fits in an exact sequence
$$0\to \Oo _X(2,2,0)\to \Ee \to \Ii _T(0,0,2)\to 0$$
with either $T=\emptyset$ or $T$ a locally complete intersection curve with multidegree $(0,0,0)$. We get $T=\emptyset$ and hence $\Ee$ fits in the sequence in (i).

Now assume $r=3$. As in the case $r=2$ we get a non-zero map $h: \Oo _X(2,2,0)\to \Ee$ with torsion-free cokernel. Let $\Gg$ be the quotient of $\Ee$ by a general map $\Oo _X\to \Ee$
with $u: \Ee \to \Gg$ the quotient map. By the case of $r=2$, the map $u\circ h : \Oo _X(2,2,0) \to \Gg$ has locally free cokernel and so $\Ff := \mathrm{coker}(h)$ is locally free. $\Ff$ is a spanned bundle of rank $2$ with no trivial factor and $c_1(\Ff )=(0,0,2)$. Thus we have $\Ff \cong \Oo _X(0,0,1)\oplus \Oo _X(0,0,1)$. Since $h^1(\Oo _X(2,2,-1)) =0$, we get $\Ee \cong \Oo _X(2,2,0)\oplus \Oo _X(0,0,1)^{\oplus 2}$.
\end{proof}

\begin{proposition}\label{v3}
Let $\Ee$ be a globally generated vector bundle of rank $2$ on $X$ with $c_1=(2,2,2)$ and $c_2=(2,2,4)$.
\begin{itemize}
\item[(i)] Its associated curve $C$ is connected if and only if $\Ee$ is a spanned flat limit of the family of Ulrich bundles in \cite[Theorem 6.7]{CFM0}.
\item[(ii)] Non-Ulrich bundles $\Ee$ exist and they are all as in Example \ref{v4}.
\end{itemize}
\end{proposition}

\begin{proof}
Since the ``~if~'' part is obvious by the description in \cite[Theorem 6.7]{CFM0}, we only need to prove the ``~only if~'' part. Since $C$ is connected, so $C \subset \PP^7$ is a smooth elliptic curve of degree $8$.

First assume that $C$ is linearly normal. The homogeneous ideal of $C$ in $\PP^7$ is generated by quadrics and so $\Ii _{C,\PP^7}(2)$ is globally generated. Thus $\Ii _C(2,2,2)$ is also globally generated and any such a curve gives a globally generated bundle. Since $h^1(\Ii _{C,\PP^7}(t)) =0$ for all $t$, we also get $h^1(\Ii _C(t,t,t)) =0$ for all $t\in \ZZ$. Since $c_1(\Ee )=(2,2,2)$, we have $\Ee ^\vee \cong \Ee (-2,-2,-2)$ and so the Serre duality gives $h^2(\Ee (t,t,t)) =0$ for all $t\in \ZZ$. Thus $\Ee$ is ACM in this case and we have the description of such a bundle in \cite[Theorem 6.7]{CFM0}.

Now assume that $C\subset \PP^7$ is not linearly normal. Since $C$ has multidegree $(2,2,4)$ and any such an embedding $C\subset \PP^1\times \PP^1\times \PP^1$ is induced by three base point free line bundles on $C$, two of degree $2$ and the other one of degree $4$, so we get that $C$ is a flat limit of a family of linearly normal elliptic curves $\{C_\lambda\}_{\lambda \in \Lambda}$ with $C_\lambda \subset X$. The Hartshorne-Serre correspondence gives that $\Ee$ is the limit of a family of $\Ee _\lambda$ with each $\Ee_\lambda$ Ulrich and in the family \cite[Theorem 6.7]{CFM0}. Since $C$ is not linearly normal, we have $h^0(\Ii _C(1,1,1)) >0$. Since $s=1$ and $e_i>0$ for all $i$, we also have $h^0(\Ii _C(0,1,1)) =h^0(\Ii _C(1,0,1)) = h^0(\Ii _C(1,1,0)) =0$. Thus $\Ee$ fits into an exact sequence (\ref{eqv2}) with either $T=\emptyset$ or a locally complete intersection curve. Since $(t_1+t_2+t_3)^2  = 2t_1t_2+2t_1t_3+2t_2t_3$, so $T$ has multidegree $(0,0,2)$, i.e. there is a $0$-dimensional subscheme $Z\subset \PP^1\times \PP^1$ of degree $2$ such that $T =Z\times \PP^1$. We get that $\Ee$ is as in  Example \ref{v4}.
\end{proof}

\begin{remark}
An example of rank two globally generated vector bundle with $c_1=(2,2,2)$ and $c_2=(2,3,3)$ is given as the first type of Ulrich bundles in \cite[Theorem 6.7]{CFM0}.
\end{remark}

\begin{proposition}\label{++}
Let $\Ee$ be a globally generated vector bundle of rank $r\ge 2$ on $X$ with $c_1=(2,2,2)$. Then we have $c_2=(0,2,2)$ if and only if $\Ee \cong \Oo _X(0,0,1)\oplus \Oo _X(2,2,1)$.
\end{proposition}

\begin{proof}
Since $t_3(2t_1+2t_2+t_3) = 2t_1t_3+2t_2t_3$, the ``~if~'' part is obvious. Take a globally generated bundle $\Ee$ with multidegree $(0,2,2)$ and let $C$ be any smooth curve which is a zero-locus of a general section of $\Ee$. By Remark \ref{v1} we have $s=1$, $C$ is an elliptic curve and there are $p\in \PP^1$ and $C'\in |\Oo _{\PP^1\times \PP^1}(2,2)|$ such that $C = \{p\}\times C'$. Since $\omega _C \cong \Oo _C$ and $s=1$, so we have $r=2$ (see Remark \ref{24nov}). Since $h^0(\Oo _C(0,0,1)) =h^0(\Oo _C)=1 < 2 = h^0(\Oo _X(2))$, we get
$H^0(\Ee (-2,-2,-1)) \ne 0$. Since $h^0(\Ii _C(0,0,0)) = h^0(\Ii _C(-1,0,1)) =h^0(\Ii _C(0,-1,1)) =0$, we get that $\Ee$ is an extension of $\Oo _X(0,0,1)$
by $\Oo _X(2,2,1)$. Since $h^1(\Oo _X(2,2,0)) =0$, so the extension is trivial.
\end{proof}

\begin{remark}\label{vo1}
Since $h^1(\Oo _X(0,0,-2)) =1$, so up to isomorphism there is a unique non-trivial extension $\Ff$ of $\Oo _X(1,1,2)$ by $\Oo _X(1,1,0)$ and we have $\Ff \cong \Oo _X(1,1,1)\oplus \Oo _X(1,1,1)$.
\end{remark}

\begin{remark}\label{bo1}
Let $W\subset \PP^7$ be a $3$-dimensional linear subspace such that $W\cap X =\emptyset$. Let $\ell _W: \PP^7\setminus W\to \PP^3$ be the linear projection from $W$ and set $\ell := {\ell _W}_{\vert_X}$. For a null-correlation bundle $\Nn_{\PP^3}(1)$ of $\PP^3$ twisted by $1$, set $\Ff := \ell ^\ast (\Nn_{\PP^3}(1))$. A general zero-locus $T$ of a section of $\Nn_{\PP^3}(1)$ is a disjoint union of two lines. Thus $\Ff$ is a globally generated bundle on $X$ with $c_1=(2,2,2)$, multidegree $(4,4,4)$ and $s=2$. Let $C$ be a general zero-locus of a section of $\Ff$. We have $h^0(\Oo _C(2,2,2)) = 24 =h^0(\Oo _X(2,2,2)) -3$. Since $h^0(\PP^3,\Ii _T(2)) =4$, we get $h^0(\Ii _C(2)) \ge 4$ and so $h^1(\Ff )>0$.
\end{remark}

\begin{lemma}\label{bo2}
Let $T\subset X$ be an elliptic curve of $\deg (T) =6$. Then the linear span $\langle T\rangle \subset \PP^7$ has dimension $5$ and $T = X\cap \langle T\rangle$ as schemes.
\end{lemma}

\begin{proof}
We have $\dim (\langle T\rangle )\le 5$ and equality holds if and only if $C$ is linearly normal in $\langle T\rangle$. Since $\PP^1\times \PP^1$ contains no elliptic curve of degree $6$, $T$ has multidegree $(2,2,2)$. Assume for the moment the existence of an irreducible surface $U\subseteq X\cap \langle T\rangle$ with $C\subset U$ and say $U\in |\Oo _X(a_1,a_2,a_3)|$. Since no component of the multidegree of $T$ has degree $0$, we have $a_i>0$ for all $i$. Thus we have $\dim (\langle T\rangle)\ge  \dim (\langle U\rangle ) \ge 6$, a contradiction.

Therefore each irreducible component of $\langle T\rangle \cap X$ has dimension at most $1$. Assume for the moment $\dim (\langle T\rangle )\le 4$. Fix a general $p\in X\setminus T$
and let $V\subset \PP^7$ be a general $5$-dimensional linear space containing $T \cup \{p\}$. Since $\deg ( T) =\deg (X)$, \cite[Theorem 2.2.5]{fov} gives a contradiction. Now assume $\dim (\langle T\rangle )=5$. Since each irreducible component of $\langle T\rangle \cap X$ has dimension at most $1$ and $\deg (T)=\deg (X)$, we get $T = X\cap \langle T\rangle$ as schemes.
 \end{proof}

 \begin{proposition}\label{bo4}
 Let $\Ee$ be a globally generated vector bundle of rank $r\ge 2$ on $X$ with $c_1=(2,2,2)$, $c_2=(2,2,2)$ and no trivial factor. Then we have $r=2$ and $\Ee \cong \Oo _X(1,1,1)^{\oplus 2}$.
 \end{proposition}

 \begin{proof}
 Since $\deg ({C}) < 8$, Remark \ref{v1} gives $s=1$ and that $C$ is a smooth elliptic curve. Since $\omega _C \cong \Oo _C$, $s=1$ and $\Ee$ has no trivial factor, we get $r=2$. Since $\deg ({C}) =6$, we have $h^0(\Ii _C(1,1,1)) \ge 2$. Since $h^1(\Oo _X(-1,-1,-1)) =0$, a non-zero section of $\Ee$ induces a non-zero map $f: \Oo _X(1,1,1)\to \Ee$. Since no entry of the multidegree of $C$ is zero, we have $h^0(\Ii _C(0,1,1)) =h^0(\Ii _C(1,0,1)) = h^0(\Ii _C(1,1,0)) =0$. Hence $f$ has torsion-free cokernel, i.e. $\mathrm{coker}(f) \cong \Ii_T(1,1,1)$ with either $T =\emptyset$ or $T$ a locally complete intersection curve. Since $c_2(\Ee )=2t_1t_2+2t_2t_3+2t_1t_3$, we get $T=\emptyset$ and so $\Ee \cong \Oo _X(1,1,1)^{\oplus 2}$.
 \end{proof}

\begin{proposition}\label{bo3}
Let $\Ee$ be a bundle of rank $2$ on $X$ with $c_1= (2,2,2)$ and the associated curve $C$ has two connected components with $\deg (C_i)=6$ for each $i$. Then $\Ee$ is globally generated if and only if it arises as in Remark \ref{bo1}, i.e. $\Ee \cong \ell ^\ast (\Nn_{\PP^3}(1))$ for some $3$-dimensional linear subspace with $W\cap X =\emptyset$ and a null-correlation bundle $\Nn_{\PP^3}(1)$ on $\PP^3$. In particular we have $h^0(\Ee )=5$ and $h^1(\Ee )=1$.
\end{proposition}

\begin{proof}
Since the ``~if~'' part is obvious, it is sufficient to prove the other implication. Call $C =C_1\cup C_2 \subset X$ be the zero-locus of a general section of $\Ee$. Setting $M_i:= \langle C_i\rangle$ for $i=1,2$, we have $M_i\cap X = C_i$ as schemes by Lemma \ref{bo2}. Set $W:= M_1\cap M_2$.
Since $\dim (M_i) = 5$ for each  $i$, so the Grassmann formula gives $\dim (W) \ge 3$. Note that we have $M_1\cap M_2\cap X =\emptyset$ since $C_1\cap C_2 =\emptyset$.
Let $\ell _W: \PP^7\setminus W\to \PP^3$ be the linear projection from $W$. Set $\ell :={ \ell _W}_{\vert_X}$, $T_i:= \ell (C_i)$ and $T:= \ell ( C)$. $T$ is the disjoint union of two lines. Call $\Nn_{\PP^3}(1)$ the twisted null-correlation bundle associated to $T$. We have $\Ii _C(2,2,2) = \ell ^\ast (\Ii _{T,\PP^3}(2))$. Since $\Ee$ is the unique bundle induced by $C$ using the Hartshorne-Serre correspondence,
we have $\Ee \cong \ell ^\ast (\Nn_{\PP^3}(1))$. The last assertion follows from our proof, since we proved that any zero-locus $C$ is the pull-back of a (obviously unique) disjoint union $T\subset \PP^3$
of two lines and hence $h^0(\Ii _C(2,2,2)) = h^0(\PP^3,\Ii _T(2)) =4$.
\end{proof}



\providecommand{\bysame}{\leavevmode\hbox to3em{\hrulefill}\thinspace}
\providecommand{\MR}{\relax\ifhmode\unskip\space\fi MR }
\providecommand{\MRhref}[2]{%
  \href{http://www.ams.org/mathscinet-getitem?mr=#1}{#2}
}
\providecommand{\href}[2]{#2}

\end{document}